\documentclass[11pt,a4paper]{article}
\usepackage[utf8]{inputenc}
\usepackage[T1]{fontenc}
\usepackage{amsmath}
\usepackage{amsfonts}
\usepackage{amssymb}
\usepackage{amsthm}
\usepackage{mathrsfs}
\usepackage{empheq}
\usepackage{enumerate}
\usepackage[backend=biber,style=alphabetic, sorting=nyt, maxalphanames=99, maxbibnames=99]{biblatex}

\usepackage{graphicx}
\usepackage[font=small,textfont=it]{caption}
\usepackage{hyperref}
\usepackage{xcolor}
\hypersetup{
    colorlinks,
    linkcolor={red!15!black!50!blue},
    citecolor={blue!50!black},
    urlcolor={blue!80!black}
}
\usepackage{todonotes}

\author{Philippe Chartier\footnote{%
Inria Rennes, IRMAR and ENS Rennes, Campus de Beaulieu, F-35170 Bruz, France. Philippe.Chartier@inria.fr
} , 
Mohammed Lemou\footnote{%
CNRS, IRMAR and ENS Rennes, Campus de Beaulieu, F-35170 Bruz, France. 
Mohammed.Lemou@univ-rennes1.fr
} , 
Léopold Trémant\footnote{%
Inria Rennes and IRMAR, Campus de Beaulieu, 35049 Rennes, France. 
Leopold.Tremant@inria.fr
}%
}
\date{}
\title{Uniformly accurate numerical schemes for a class of dissipative systems}

\newcommand{\eps}{\varepsilon}
\newcommand{\N}{\mathbb{N}}
\newcommand{\R}{\mathbb{R}}
\newcommand{\C}{\mathbb{C}}
\newcommand{\bigO}{\mathcal{O}}

\newcommand{\ol}[1]{\overline{#1}}
\newcommand{\eeps}{^{\eps}}

\newcommand{\inveps}{\frac{1}{\eps}}

\newcommand{\dpt}{\partial_t}
\newcommand{\dptau}{\partial_{\tau}}
\newcommand{\dptheta}{\partial_{\theta}}
\newcommand{\dpu}{\partial_{u}}
\newcommand{\dpx}{\partial_{x}}

\newcommand{\Dt}{\Delta t}
\newcommand{\D}{\mathrm{d}}
\newcommand{\id}{\mathrm{id}}

\newcommand{\among}[2]{\begin{pmatrix} #1 \\ #2 \end{pmatrix}}

\newcommand{\vertiii}[1]{{\left\vert\kern-0.25ex\left\vert\kern-0.25ex\left\vert #1 %
    \right\vert\kern-0.25ex\right\vert\kern-0.25ex\right\vert}}

\newtheorem{zzzzzz}{ }[section]
\newtheorem{theorem}[zzzzzz]{Theorem}
\newtheorem{proposition}[zzzzzz]{Proposition}

\newtheorem{remark}[zzzzzz]{Remark}

\newtheorem{lemma}[zzzzzz]{Lemma}
\newtheorem{assumption}[zzzzzz]{Assumption}
\newtheorem{property}[zzzzzz]{Property}

\numberwithin{equation}{section}

\begin{document}

\maketitle
\renewcommand{\tilde}{\widetilde}

\begin{abstract}
We consider a class of relaxation problems mixing slow and fast variations 
which can describe population dynamics models or hyperbolic systems, 
with varying stiffness (from non-stiff to strongly dissipative), 
and develop a multi-scale method by decomposing this problem into a micro-macro system 
where the original stiffness is broken. 
We show that this new problem can therefore be simulated 
with a \textit{uniform} order of accuracy using standard explicit numerical schemes. 
In other words, it is possible to solve the micro-macro problem 
with a cost independent of the stiffness (a.k.a. uniform cost), such that the error is also uniform. 
This method is successfully applied to two hyperbolic systems 
with and without non-linearities, 
and is shown to circumvent the phenomenon of order reduction. 
\\

\\

\end{abstract}

\section{Introduction}

We are interested in problems of the form, for $x\eeps(t) \in \R^{d_x}$ and $z\eeps(t) \in \R^{d_z}$, 
\begin{equation} \label{intro-eq:xz_pb}
\left\{
\begin{array}{ll} \displaystyle
\dot{x}\eeps = a(x\eeps,z\eeps), & x\eeps(0) = x_0 , 
\\ \displaystyle
\dot{z}\eeps = -\inveps A z\eeps + b(x\eeps, z\eeps), \quad & z\eeps(0) = z_0 , 
\end{array} \right.
\end{equation}
with $\eps \in (0,1]$ a small parameter, 
$A$ a diagonal positive matrix with integer coefficients, 
and where $a,b$ are respectively the $x$-component and the $z$-component of an analytic map $f$ which smoothly depends on $\eps$. 
In the sequel we shall more often write this problem as 
\begin{equation} \label{intro-pb:full_pb_on_u}
\dot{u}\eeps = -\inveps \Lambda u\eeps + f(u\eeps), \quad u\eeps(0) = u_0,
\end{equation}
where $u = \among xz$, $\Lambda = \begin{pmatrix}
0 & 0 \\ 0 & A
\end{pmatrix}$ and $f(u) = \among{a(x,z)}{b(x,z)}$. 
We set $d = d_x + d_z$ the dimension of $u$ such that $u \in \R^d$. 
In particular, the dimension of $x\eeps$ can be zero without impacting our results. 
The map $u \mapsto f(u)$ is assumed to be smooth. 
Our theorems do not consider the case where $f$ involves a differential operator in space 
(i.e. the case of partial differential equations). 
Nonetheless, two of our examples are discretized hyperbolic partial differential equations (PDEs) 
for which the method is successfully applied, 
even though a special treatment is required. 

Systems of this kind appear in population dynamics 
(see \cite{greiner1994singular, auger1996emergence, sanchez2000singular, castella2015analysis}), 
where $A$ accounts for migration (in space and/or age) 
and $a,b$ account for both the demographic and inter-population dynamics.
The migration dynamics is quantifiably faster than the other dynamics involved, 
which explains the rescaling by $\eps$ in the model. 
When solving this kind of system numerically, 
problems arise due to the large range of values that $\eps$ can take.

Considering a numerical scheme of order $q > 1$, by definition, for all $\eps$, 
there exists a constant $C(\eps)$ and a time-step $\ol{\Dt}(\eps)$ 
such that for all $\Dt < \ol{\Dt}(\eps)$, the error $E_{\eps}(\Dt)$ when solving~\eqref{intro-pb:full_pb_on_u} is bounded by 
$$
E_{\eps}(\Dt) \leq C(\eps) \Dt^{q} .
$$
Assume now that there exists $\Dt^*$ such that this scheme is stable for all $\eps \in (0,1]$ and $\Dt < \Dt^*$.%
\footnote{ 
In particular, the scheme cannot be any usual explicit scheme 
since it would require a stability condition of the form $\Dt/\eps < C$ with $C$ independent of $\eps$. } 
The order reduction phenomenon manifests itself through the existence of $s < q$ and $C > 0$, 
both independent of $\eps$ such that the uniform error $E(\Dt) := \sup_{\eps} E_{\eps}(\Dt)$ satisfies 
\begin{equation} \label{intro_eq:unif_error}
\sup_{\eps \in (0,1]} E_{\eps}(\Dt) \leq C \Dt^s . 
\end{equation} 
Note that in general $s$ is much smaller than $q$. 
This behaviour is documented for instance in~\cite[Section~IV.15]{hairer1996stiff} or in~\cite{hundsdorfer2007imex}. 
In order to ensure a given error bound, one must either accept this order reduction (if $s > 0$), 
as is done for asymptotic-preserving (AP) schemes~\cite{jin1999efficient} 
by taking a modified time-step $\tilde{\Dt} = \Dt^{q/s}$, 
or use an $\eps$-dependent time-step $\Dt = \bigO(\eps^{\alpha})$ for some $\alpha > 0$. 
In practice, both approaches cause the computational cost of the simulation to increase greatly, often prohibitively so. 

Another common approach to circumvent this is to invoke 
the \textit{center manifold theorem} (see~\cite{vasil1963asymptotic, carr1982, sakamoto1990invariant}) 
which dictates the long-time behaviour of the system 
and presents useful characteristics for numerical simulations: 
the dimension is reduced and the dynamics on the manifold is non-stiff. 
However, this approach does not capture the \textit{transient solution} of the problem, 
i.e. the solution in short time before it reaches the stable manifold. 
This is troublesome when one wishes to describe the system out of equilibrium. 
Furthermore, even if the solution is close to the manifold, 
these approximations are accurate up to a certain order $\bigO(\eps^n)$, 
rendering them useless if $\eps$ is of the order of $1$. 

We first provide a systematic way to compute asymptotic models at any order in $\eps$ 
that approach the solution \textit{even in short time}. 
Then we use the defect of this approximation to compute the solution 
with usual explicit numerical schemes and \textit{uniform} accuracy 
(i.e. the cost and error of the scheme must be independent of $\eps$).
This approach automatically overcomes the challenges posed by both 
extremes $\eps \ll 1$ and $\eps \sim 1$. 
\\

In order to achieve this goal, for any non-negative integer $k$ we construct 
a change of variable for the dissipative problem~\eqref{intro-pb:full_pb_on_u}, 
$(\tau,u) \in \R_+ \times \R^d \mapsto \Omega^{[n]}_{\tau}(u) \in \R^d$, 
and a \textit{non-stiff} vector field $u\in \R^d \mapsto F^{[n]}(u) \in \R^d$, 
such that 
\begin{equation} \label{intro_eq:mima_decomp}
u\eeps(t) = \Omega^{[n]}_{t/\eps} \left( v^{[n]}(t) \right) + w^{[n]}(t) 
\end{equation} 
where $v^{[n]}$ is the \textit{macro} component with dynamics dictated by $F^{[n]}$, 
and $w^{[n]}$ is the \textit{micro} component of size $\bigO(\eps^{n+1})$. 
The main result we prove is that from this decomposition, 
it is possible to compute $u\eeps$ with \textit{uniform accuracy} 
when using explicit exponential Runge-Kutta schemes of order $n+1$ 
(which can be found for instance in~\cite{cite:exp_erk_schm}), 
i.e. it is possible to take $s = q = n+1$ in~\eqref{intro_eq:unif_error}. 
In other words, if $(t_i)_{0 \leq i \leq N}$ is a discretisation of time-step $\Delta t$, 
and $(v_i)$ and $(w_i)$ are computed numerically using such a scheme, 
then there exists \textit{$C$ independent of $\eps$} such that 
$$ 
\max_{0 \leq i \leq N} \left| u\eeps(t_i) - \Omega^{[n]}_{t_i/\eps}(v_i) - w_i \right| \leq C \Dt ^{n+1} 
$$
where $| \cdot |$ is the usual Euclidian norm on $\R^d$. 
Furthermore, using a scheme of order $n$ generates an error proportional to $\eps$ 
on the $z$-component of the solution. 
This is interesting as $z\eeps$ is of size $\eps$ after a time $\bigO(\eps \log(1/\eps))$. 
IMEX methods such as CNLF and SBDF (see~\cite{ascher1995implicit, akrivis1999implicit, hu2019uniform}), 
which mix implicit and explicit solving (for the stiff and non-stiff part respectively) 
are not the focus of the article, 
but their use is briefly discussed in Remark~\ref{mima_rq:imex}. 
\\

Recently in~\cite{cite:asym_bseries}, asymptotic expansions of the solution of~\eqref{intro-eq:xz_pb} were constructed in the case $A = I_{d_z} , $ 
allowing an approximation of the solution of~\eqref{intro-eq:xz_pb} 
with an error of size $\bigO(\eps^{n+1})$.  
This method could be considered to compute the change of variable~$(\tau,u) \mapsto \Omega^{[n]}_{\tau}(u)$. 
However it involves elementary differentials and manipulations on trees 
which are impractical to implement, especially for higher-orders. 
For highly-oscillatory problems, another approach, developed in~\cite{cite:mic_mac}, 
involves a recurrence relation which could later be computed automatically for high orders~\cite{chartier2020high_order}. 
We start by considering the following problem 
\begin{equation} \label{intro_pb:oscill}
\dot y\eeps = -i e^{-i\frac{t}{\eps}\Lambda} f \left( e^{i\frac{t}{\eps} \Lambda} y\eeps \right), \qquad y\eeps(0) = y_0 := u_0 
\end{equation}
on which we apply averaging methods detailed in~\cite{cite:strobo} 
that are in the vein of those initiated by~\cite{perko1969higher} 
in order to approach the solution with the composition of a near-identity periodic map $(\theta, u) \mapsto \Phi^{[n]}_{\theta}(u)$ 
and a flow $(t,u) \mapsto \Psi^{[n]}_t(u)$ following a vector field $G^{[n]}$: 
$ y\eeps(t) = \Phi^{[n]}_{t/\eps} \circ \Psi^{[n]}_t \circ \big( \Phi^{[n]}_0 \big)^{-1}(y_0) + \tilde{y}^{[n]}(t) $
for all $n \geq 0$, where $\tilde{y}^{[n]}$ is of size $\bigO(\eps^{n+1})$ 
and can be computed numerically with a uniform error. 
The change of variable $\Omega^{[n]}$ and the vector field $F^{[n]}$ 
are then deduced from $\Phi^{[n]}$ and $G^{[n]}$ using Fourier series. 
From this, the micro-macro problem defining $v^{[n]}$ and $w^{[n]}$ in~\eqref{intro_eq:mima_decomp} 
for the dissipative problem~\eqref{intro-pb:full_pb_on_u} is deduced. 
\\

The rest of the paper is organized as follows. 
In Section~\ref{sec:avg}, we construct the change of variable and smooth vector field 
used to obtain the macro part in~\eqref{intro_eq:mima_decomp} for Problem~\eqref{intro-pb:full_pb_on_u}. 
These maps are constructed using averaging methods on~\eqref{intro_pb:oscill} 
and properties similar to those of averaging are proven, 
ensuring the well-posedness of the micro-macro equations on $(v^{[n]}, w^{[n]})$ 
as defined in~\eqref{intro_eq:mima_decomp}. 
In Section~\ref{sec:mima}, we study the micro-macro problems associated with 
this new decomposition~\eqref{intro_eq:mima_decomp}, and prove that the micro part $w^{[n]}$ 
is indeed of size $\eps^{n+1}$, and that the problem is not stiff. 
We then state the result of uniform accuracy when using exponential RK schemes. 
In Section~\ref{sec:pde}, we present some techniques 
to adapt our method to discretized hyperbolic PDEs. 
Namely, we study a relaxed conservation law and the telegraph equation, 
which can be respectively found for instance in~\cite{jin1995relaxation} 
and~\cite{lemou2008asymptotic}. 
In Section~\ref{sec:tests}, we verify our theoretical result of uniform accuracy 
by successfully obtaining uniform convergence when numerically solving 
micro-macro problems obtained from a toy ODE 
and from the two aforementioned PDEs. 
\section{Derivation of asymptotic models with error estimates} \label{sec:avg}
\newcommand{\K}{\mathcal{K}}

In this section, we construct the change of variable $(\tau,u) \in \R_+ \times \R^d \mapsto \Omega^{[n]}_{\tau}(u)$ and vector field $u \in \R^d \mapsto F^{[n]}(u)$ 
used in the micro-macro decomposition~\eqref{intro_eq:mima_decomp}. 
In Subsection~\ref{sec:approx:subsec:assump}, assumptions on the vector field $u \mapsto f(u)$ 
and on the solution~$u\eeps$ of~\eqref{intro-pb:full_pb_on_u} are stated. 
In Subsection~\ref{sec:approx:subsec:oscill}, we define a highly-oscillatory problem 
and construct an asymptotic approximation of the solution of this problem as in~\cite{cite:mic_mac}. 
We finish the subsection by summarizing the error bounds associated to this approximation. 
In Subsection~\ref{sec:approx:subsec:dissip}, we finally define $\Omega^{[n]}$ and $F^{[n]}$, 
and state results on error bounds akin to those in the highly-oscillatory case. 
While these are asymptotic expansions, the error bounds are valid 
for all values of $\eps$, so that the micro-macro decomposition~\eqref{intro_eq:mima_decomp} 
is always valid.

\subsection{Definitions and assumptions} \label{sec:approx:subsec:assump}

In order for the highly-oscillatory problem~\eqref{intro_pb:oscill} to be well-defined,  
we first make the following assumption. 
\begin{assumption} \label{approx_hyp:f_poly}
Let us set $d = d_x + d_z$ the dimension of Problem~\eqref{intro-pb:full_pb_on_u}. 
There exists a compact set $X_1 \subset \R^{d_x}$ and a radius $\check{\rho} > 0$ 
such that for every $x$ in $X_1$, the map $u \in \R^d \mapsto f(u) \in \R^d $ 
can be developed as a Taylor series around $\among{x}{0}$, 
and the series converges with a radius not smaller than $\check{\rho}$. 
\end{assumption}

It is therefore possible to naturally extend $f$ to closed subsets of $ \C^d $ 
defined by  
\begin{equation*} \label{approx_def:U_rho} 
\mathcal{U}_{\rho} := \left\{ u \in \C^d \,;\, \exists x \in X_1,\, \left| u - \among{x}{0_{d_z}} \right| \leq \rho \right\} , 
\end{equation*} 
for all $0 \leq \rho < \check{\rho}$ as it is represented by a Taylor series in $u\in \C^d$ on these sets. 
Here $| \cdot |$ is the natural extension of the Euclidian norm on $\R^d$ to $\C^d$.

It may seem particularly restrictive to assume that the $z$-component of the solution $u\eeps$ of~\eqref{intro-pb:full_pb_on_u} 
stays in a neighborhood of $0$, however this is somewhat ensured by the \textit{center manifold theorem}. 
This theorem states that there exists a map $x \in \R^{d_x} \mapsto \eps h\eeps(x) \in \R^{d_x}$ smooth in $\eps$ and $x$, 
such that the manifold $\mathcal{M}$ defined by 
$$ \mathcal M = \left\{ (x,z) \in \R^{d_x} \times \R^{d_z} \: : \: z = \eps h\eeps(x) \right\} $$
is a stable invariant for~\eqref{intro-eq:xz_pb}. 
It also states that all solutions $(x\eeps, z\eeps)$ of~\eqref{intro-eq:xz_pb} 
converge towards it exponentially quickly, 
i.e. there exists $\mu > 0$ independent of $\eps$ such that 
\begin{equation} \label{approx_eq:cent_manifold}
\left| z\eeps(t) - \eps h\eeps(x\eeps(t)) \right| \leq C e^{-\mu t/\eps} .
\end{equation} 
This means that the growth of $z\eeps$ is bounded by that of $x\eeps$, 
and that after a time $t \geq \eps \log(1/\eps)$, $z\eeps(t)$ is of size $\bigO(\eps)$. 
Therefore it is credible to assume that $z\eeps$ stays somewhat close to~$0$. 
This is translated into a second assumption. 

\begin{assumption} \label{approx_hyp:well_posed} 
There exist two radii~$0 < \rho_0 \leq \rho_1 < \check{\rho}$ 
and a closed subset $X_0 \subset X_1 \subset \R^{d_x}$ 
such that the initial condition $ u_0 \in \C^d $ satisfies 
$$ 
\min_{x \in X_0} \left| u_0 - \among{x}{0_{d_z}} \right| \leq \rho_0 , 
$$ 
and for all $\eps \in (0,1]$, 
Problem~\eqref{intro-pb:full_pb_on_u} is well-posed on $[0,1]$ 
with its solution $u\eeps$ in $\mathcal{U}_{\rho_1}$. 
\end{assumption} 
Note that this is different to assuming that the initial data $(x_0, z_0)$ is close to the center manifold. 
For $\rho \in [0, \check{\rho} - \rho_1),$ we define the sets 
\begin{equation} \label{approx_def:set_K}
\K_{\rho} := \mathcal{U}_{\rho_1 + \rho} = \left\{ u \in \C^d \,;\, \exists x \in X_1, \left| u - \among{x}{0} \right| \leq \rho_1 + \rho \right\} 
\end{equation} 
which help quantify the distance to the solution $u\eeps$. 
By Assumption~\ref{approx_hyp:well_posed}, the solution of~\eqref{intro-pb:full_pb_on_u} 
is in $\K_0$ at all time. 
\subsection{Constructing an approximation of the periodic problem} \label{sec:approx:subsec:oscill}
\newcommand{\T}{\mathbb{T}}

Writing $\mathbb T = \R/2\pi\mathbb Z$, 
we define a map $ (\theta, u) \in \mathbb T\times \mathcal{U}_{\check{\rho}} \mapsto g_{\theta}(u) \in \C^d $ by 
\begin{equation} \label{approx_def:g} 
g_{\theta}(u) := -ie^{-i\theta \Lambda} f \left( e^{i\theta \Lambda}u \right) . 
\end{equation} 
Thanks to Assumption~\ref{approx_hyp:f_poly}, $g$ is well-defined 
and is analytic w.r.t. both~$\theta$ and~$u$. 
In this subsection, we consider the highly-oscillatory problem 
\begin{equation} \label{approx-pb:periodic} 
\dot{y}\eeps = g_{t/\eps}(y\eeps), \qquad y\eeps(0) = y_0 := u_0 ,
\end{equation} 
of which we approach the solution using averaging techniques based on a recurrence relation from~\cite{cite:strobo}. 
The following construction and results are taken from \cite{cite:mic_mac}, 
where they are described in (much) more detail 
and where they serve to construct the macro-part of a micro-macro decomposition of~$y\eeps$. 
We start by writing the solution of~\eqref{approx-pb:periodic} as a composition 
\begin{equation} \label{approx-eq:periodic_decomposition}
y\eeps(t) = \Phi_{t/\eps}^{[n]} \circ \Psi_t^{[n]} \circ \big( \Phi_0^{[n]} \big)^{-1}(y_0) + \bigO(\eps^{n+1}) 
\end{equation}
where $\Phi^{[n]}$ is a change of variable $(\theta,u) \in \mathbb T\times \mathcal{U}_{\check{\rho}} \rightarrow \Phi\eeps_{\theta}(u) \in \C^d$ and $\Psi^{[n]}$ is the flow map of an autonomous differential equation
$$ \frac{\D}{\D t} \Psi_t^{[n]}(u) = G^{[n]} \Big( \Psi_t^{[n]}(u) \Big), \qquad \Psi_0^{[n]} = \id $$
where $G^{[n]}$ is a smooth map which must be determined. 

The idea behind this composition is that $\Psi^{[n]}$ captures the slow drift while $\Phi^{[n]}$ captures rapid oscillations. 
In this work, we focus on standard averaging, 
meaning the change of variable is of identity average, 
i.e. $\langle \Phi^{[n]} \rangle = \id$. 
The average is defined by 
\begin{equation} \label{approx_eq:def_avg}
\langle \varphi \rangle (u) := \frac{1}{2\pi} \int_{-\pi}^{\pi} \varphi_{\theta}(u) \D\theta. 
\end{equation}
The change of variable $\Phi^{[n]}$ is computed iteratively using the relation 
\begin{equation} \label{approx-eq:periodic_iter}
\Phi^{[n+1]}_{\theta} = \id + \eps \int_0^{\theta} T(\Phi^{[n]})_{\sigma}\D\sigma - \eps \left\langle \int_0^{\bullet} T(\Phi^{[n]})_{\sigma} \D\sigma \right\rangle 
\end{equation}
with initial condition $\Phi^{[0]} = \id$. 
The operator $T$ is defined for maps $(\theta,u) \mapsto \varphi_{\theta}(u)$ with identity average as 
\begin{equation} \label{approx_def:operator_T}
T(\varphi)_{\sigma} = g_{\sigma} \circ \varphi_{\sigma} - \dpu \varphi_{\sigma} \cdot \langle g \circ \varphi \rangle . 
\end{equation}
From these changes of variable $\Phi^{[n]}$ we define vector fields $G^{[n]}$ and defects $\delta^{[n]}$ by
\begin{equation} \label{approx_eq:def_G_delta}
G^{[n]} := \langle g\circ \Phi^{[n]} \rangle, \qquad \delta^{[n]} := \inveps \dptheta \Phi^{[n]} + \dpu \Phi^{[n]} G^{[n]} - g\circ \Phi^{[n]}. 
\end{equation}
Note that by definition, $\langle \delta^{[n]} \rangle = 0$.

Given a radius $\rho \geq 0$ and a map $(\theta,u)\in \mathbb T \times \K_{\rho} \mapsto \varphi_{\theta}(u)$ analytic in $u$ and $\nu$-times continuously differentiable in $\theta$, we define the norms
\begin{equation} \label{approx_eq:def_norms_perio}
\| \varphi \|_{\T, \rho} := \sup_{(\theta, u) \in \mathbb T \times \K_{\rho}} | \varphi_{\theta}(u) |,
\qquad
\| \varphi \|_{\T, \rho,\nu} := \sup_{0 \leq \beta \leq \nu} \| \dptheta^{\:\beta} \varphi \| _{\T, \rho} .
\end{equation}
We later use these norms to state error bounds on maps $\Phi^{[n]}$ and $\delta^{[n]}$. 

\begin{property}
\label{approx_prop:assump_perio}
Assumptions~\ref{approx_hyp:f_poly} and~\ref{approx_hyp:well_posed}
ensure the following properties: 
\begin{enumerate}[(i)]
\setlength{\itemsep}{0pt}
\item There exists a final time $T > 0$ such that for all $\eps \in (0,1]$, 
Problem~\eqref{approx-pb:periodic} is well-posed on $[0, T]$ 
with its solution~$y\eeps$ in $\K_0$. 
\item There exists a radius $R > 0$ such that for all $\theta \in \mathbb T$, the function $u \mapsto g_{\theta}(u) $ is analytic from $\K_{2R}$ to $\C^d$.
\item As the function $(\theta, u) \mapsto g_{\theta}(u)$ is analytic w.r.t. $\theta$, 
we fix an arbitrary rank $p > 0$ 
and set $M > 0$ a constant such that for all $\sigma \in [0, 3]$, 
\begin{equation} \label{approx_def:cst_M}
\forall\ 0\leq \nu \leq p+2, \quad \frac{\sigma^{\nu}}{\nu!} \left\| \dptheta^{\:\nu} g \right\|_{\T, 2R} \leq M , 
\end{equation} 
\end{enumerate} 
\end{property}
\noindent%
This allows us to get averaging results which can be summed up in the following theorem: 

\begin{theorem}[from \cite{cite:mic_mac} and \cite{cite:strobo}] \label{approx-thm:period_bounds}
For $n \in \N^*$, let us denote $r_n = R/n$ and $\eps_n := r_n/16 M$ 
with $R$ and $M$ defined in Property~\ref{approx_prop:assump_perio}. 
For all $\eps > 0$ such that $\eps \leq \eps_n$, the maps $\Phi^{[n]}$ and $G^{[n]}$ are well-defined 
by~\eqref{approx-eq:periodic_iter} and~\eqref{approx_eq:def_G_delta}. 
The change of variable $\Phi^{[n]}$ and the defect $\delta^{[n]}$ are both 
$(p+2)$-times continuously differentiable w.r.t. $\theta$, 
and $\Phi_0^{[n]}$ is invertible with analytic inverse on $\K_{R/4}$. 
Moreover, the following bounds are satisfied for $0 \leq \nu \leq p+1$,
\begin{align*}
(i)\quad{}&{} \| \Phi^{[n]} - \id \|_{\T, R} \leq 4\eps M \leq \frac{r_n}{4} ,
&  
(ii)\quad{}&{} \|\dptheta^{\:\nu} \Phi^{[n]} \|_{\T, R} \leq 8\eps M \nu! 
\\
(iii)\quad{}&{} \| G^{[n]} \|_{\T, R} \leq 2M 
& 
(iv)\quad{}&{} \|\delta^{[n]} \|_{\T, R,p+1} \leq 2M \left( 2\mathcal Q_p \frac{\eps}{\eps_n} \right)^n 
\end{align*}
where $\mathcal Q_{p}$ is a $p$-dependent constant. 
\end{theorem}

These properties ensure that the micro-macro problem is well-posed in~\cite{cite:mic_mac}. 
We now use these maps $\Phi^{[n]}$, $G^{[n]}$ and $\delta^{[n]}$ in order to define a decomposition for the dissipative problem~\eqref{intro-pb:full_pb_on_u}, 
and show that similar properties are satisfied.

\subsection{A new decomposition in the dissipative case}
\label{sec:approx:subsec:dissip}

A map $(\theta, u) \in \mathbb{T} \times \K_{\rho} \mapsto \varphi_{\theta}(u)$ 
which is continuously differentiable w.r.t. $\theta$ 
coincides everywhere with its Fourier series, i.e. 
\begin{equation} \label{approx_def:fourier_coeff}
\varphi_{\theta}(u) = \sum_{j\in\mathbb Z} e^{ij\theta} c_j\big( \varphi \big)(u), 
\ \ \text{where}\ \ 
c_j \big( \varphi \big)(u) = \frac{1}{2\pi} \int_{-\pi}^{\pi} e^{-ij\theta} \varphi_{\theta}(u) \D\theta. 
\end{equation} 
We define the shifted map $\tilde{\varphi}$ by 
\begin{equation} \label{approx_def:shift_operator}
\tilde{\varphi}_{\theta}(u) = e^{i\theta \Lambda} \varphi_{\theta}(u) . 
\end{equation}
Using these Fourier coefficients $(c_j)_{j\in \mathbb Z}$, 
we consider new maps by setting the change of variable $\Omega^{[n]}$ and the defect $\eta^{[n]}$, for $(\tau, u) \in \R_+ \times \K_{R_k}$, 
\begin{equation} \label{approx-eq:def_diff_var_chg}
\Omega_{\tau}^{[n]}(u) := \sum_{j \in \mathbb Z} e^{-j\tau} c_j \big( \widetilde{\Phi}^{[n]} \big)(u) ,
\qquad 
\eta_{\tau}^{[n]}(u) := i \sum_{j \in \mathbb Z} e^{-j\tau} c_j \big( \widetilde{\delta}^{\,[n]} \big)(u) .
\end{equation}
These series are purely formal for now, and their convergence is demonstrated at the end of this subsection. 
Here $\widetilde{\Phi}^{[n]}$ and $\widetilde{\delta}^{[n]}$ are respectively 
the shifted change of variable and the shifted defect, 
with the shift given by~\eqref{approx_def:shift_operator}. 
If there exists an index $j < 0$ and a vector $u \in \K_{\rho}$ 
such that $c_j(\widetilde{\Phi}^{[n]})(u) \neq 0$, 
then $\Omega^{[n]}_{\tau}(u)$ cannot be bounded uniformly for all $\tau \in \R_+$. 
We also define the flow $\Gamma^{[n]}$ by setting 
\begin{equation} \label{approx-eq:diff_flow}
\frac{\D}{\D t} \Gamma_t^{[n]}(u) = F^{[n]} \Big( \Gamma^{[n]}_t(u) \Big), 
\quad \text{where}\quad
F^{[n]} = i G^{[n]} .
\end{equation}
Note that we do not know the lifetime of any particular solution of the Cauchy problem $\dpt v = F^{[n]}(v)$, $v(0) = v_0 \in \K_R$ yet.

\begin{remark}
From the identity $g_{\theta}(\ol{u}) = - \ol{g_{-\theta}(u)}$, 
one can obtain the relations on the Fourier coefficients 
\begin{equation} \label{approx_eq:maps_are_real}
c_j(\Phi^{[n]})(\ol{u}) = \ol{c_j(\Phi^{[n]}) (u)} 
\quad\text{and}\quad
c_j(\delta^{[n]})(\ol{u}) = -\ol{c_j(\delta^{[n]}) (u)} . 
\end{equation} 
The same holds for $\tilde{\Phi}^{[n]}$ and $\tilde{\delta}^{[n]}$, 
as the tilde operator simply shifts the indices of these coefficients 
component by component. 
This ensures that if $u$ is in $\R^d$ then so are $\Omega^{[n]}_{\tau}(u)$ and $\eta^{[n]}_{\tau}(u)$. 
Similarly, if $u$ is in $\R^d$ then $F^{[n]}(u)$ is in $\R^d$. 

For $\Phi^{[n]}$, for instance, it is equivalent to $\Phi^{[n]}_{\theta}(\ol{u}) = \ol{\Phi^{[n]}_{-\theta}(u)}$, 
and can be shown by induction using~\eqref{approx-eq:periodic_iter}. 
Indeed, for $j = 0$, $c_j(\Phi^{[n]})(u) = u$ (making the result straightforward), 
and for $j \neq 0$ the change of variable satisfies 
$c_j(\Phi^{[n+1]})(u) = \frac{1}{ij} c_j \left( T(\Phi^{[n]}) \right) (u)$ with $T(\Phi^{[n]})$ defined by~\eqref{approx_def:operator_T}. 
The following calculation is also valid for $j=0$, 
\begin{align*}
c_j(g\circ \Phi^{[n]})(\ol{u}) 
&= \frac{1}{2\pi} \int_{-\pi}^{\pi} e^{-ij\theta} g_{\theta}\left( \Phi^{[n]}_{\theta} (\ol{u}) \right) \D\theta 
= \frac{1}{2\pi} \int_{-\pi}^{\pi} e^{-ij\theta} g_{\theta}\left( \ol{ \Phi^{[n]}_{-\theta} (u) } \right) \D\theta 
\\
&= -\frac{1}{2\pi} \int_{-\pi}^{\pi} \ol{ e^{ij\theta} g_{-\theta} \left( \Phi^{[n]}_{-\theta}(u) \right) } \D \theta 
= -\ol{c_j( g\circ \Phi^{[n]})(u)} .
\end{align*}
From this, one gets $c_j\left( T(\Phi^{[n]}) \right) (\ol{u}) 
= - \ol{c_j(g\circ \Phi^{[n]})(u)} - \ol{\dpu c_j(\Phi^{[n]})(u)} \cdot \left( - \ol{c_0( g\circ \Phi^{[n]})(u)} \right) 
= - \ol{c_j \left( T(\Phi^{[n]}) \right) (u)} $, 
yielding the desired result. 
\end{remark}

The micro part $w^{[n]}$ of the decomposition is the difference between 
the solution $u\eeps$ of~\eqref{intro-pb:full_pb_on_u} 
and the asymptotic approximation $ \Omega_{t/\eps}^{[n]} \circ \Gamma_t^{[n]} \circ \big( \Omega_0^{[n]} \big)^{-1}(u_0) $. 
Assuming that $\Omega^{[n]}$ and $\eta^{[n]}$ are well-defined 
(this is proved it in Theorem~\ref{approx_thm:diff_cov_bounds}), 
it is necessary to show that the map $\eta^{[n]}$ can be characterized 
as a \textit{defect} (similarly to $\delta^{(k]}$). 
Being a defect means that $\eta^{[n]}$ characterises the error 
of the approximation $\frac{\D}{\D t} \left[ \Omega^{[n]}_{t/\eps} \circ \Gamma^{[n]}_t \right] 
\approx -\inveps \Lambda \Gamma^{[n]}_t + f \circ \Omega^{[n]}_{t/\eps} \circ \Gamma^{[n]}_t $. 
A straightforward computation yields 
\begin{equation} \label{approx_eq:other_def_eta} 
\eta^{[n]}_{\tau} = \sum_{j \in \mathbb Z} e^{-j\tau} c_j \left( \frac{i}{\eps} \widetilde{\dptheta \Phi} ^{[n]} + \widetilde{\dpu \Phi} ^{[n]} \cdot (i G^{[n]}) - i \widetilde{g \circ \Phi}^{[n]} \right)
\end{equation}
where we can recognize $\widetilde{\dptheta \varphi} = \dptheta \widetilde{\varphi} - i \Lambda \widetilde{\varphi}$, 
$\tilde{\dpu \varphi} = \dpu \tilde{\varphi}$ 
and $i \widetilde{g\circ \varphi} = f \circ \widetilde{\varphi}$. 
The characterization as a defect requires the following result:

\begin{lemma} \label{approx_lemma:commut_cj}
Let $\rho$ and $r$ be two radii such that $ 0 \leq \rho < r \leq 2R $ 
and let $\nu$ be a positive integer. 
We set $\varphi$ a periodic map $(\theta, u) \in \mathbb{T} \times \K_{\rho} \mapsto \varphi_{\theta}(u) \in \K_r $ that is near-identity in the sense  
$$
\forall (\theta, u) \in \mathbb{T} \times \K_{\rho} , \quad 
\left| \varphi_{\theta}(u) - u \right| \leq r - \rho 
$$ 
and that is continuously differentiable w.r.t. $\theta$ for all $u \in \K_{\rho}$. 
With the definitions of~\eqref{approx_def:fourier_coeff} and~\eqref{approx_def:shift_operator}, 
assume that all the Fourier coefficients of negative index of the shifted map $\tilde{\varphi}$ vanish. 
Then, setting $D := \{ \xi \in \C ,\ |\xi| < 1 \}$ and $\ol{D}$ its closure, 
the map $ (\xi, u) \in \ol{D} \times \K_{\rho} \mapsto \sum_{j \in \mathbb{Z}} \xi^j c_j \left( \tilde{\varphi} \right) (u) $ 
is well-defined with values in $\K_{r}$, $p$-times continuously differentiable. 
Furthermore, for all $(\xi, u) \in \ol{D} \times \K_{\rho}$, 
when composing with the vector field $u \mapsto f(u)$ from~\eqref{intro-pb:full_pb_on_u} (satisfying Assumption~\ref{approx_hyp:f_poly}), 
the following identity is met 
$$ 
f \left( \sum_{j \geq 0} \xi^j c_j ( \tilde{\varphi} ) (u) \right) 
= \sum_{j \in \mathbb{Z} } \xi^j c_j \left( f \circ \tilde{\varphi} \right) (u) 
$$ 
and for all $ j < 0 $, $c_j \left( f \circ \tilde{\varphi} \right)(u)$ is identically zero. 
In particular the map $(\tau, u) \in \R_+ \times \K_{\rho} \mapsto \sum_{j \in \mathbb Z} e^{-j\tau} c_j(\tilde{\varphi})(u) $ 
is well-defined with values in $\K_r$. 
\end{lemma} 

\begin{proof}
\newcommand{\fcircPhi}{f \circ \tilde{\Phi}^{[k]}}
Let us work at fixed $u \in \K_{\rho}$. 
By product $\tilde{\varphi}$ is continuously differentiable w.r.t. $\theta$, 
therefore the series of its Fourier coefficients is absolutely convergent. 
Furthermore, $\tilde{\varphi}$ only has nonnegative modes by assumption, 
meaning the indices can be restricted to nonnegative values 
in the definition 
$$
\zeta : \xi \in \ol{D} \mapsto \zeta (\xi) 
:= \sum_{j \in \mathbb{Z} } \xi^j c_j\left( \tilde{\varphi} \right)(u) . 
$$ 
As such, the function $\zeta$ is well-defined on $\ol{D}$ and is holomorphic on $D$.

Let us now show that it has values in $\K_r$. 
Because $u$ is in $\K_{\rho}$, by~\eqref{approx_def:set_K}, 
we set $x \in X_1$ such that 
$
\left| u - \among{x}{0} \right| \leq \rho_1 + \rho . 
$ 
Using a triangle inequality in the definition of~$\zeta$ yields 
\begin{equation} \label{approx_eq:lemma:bound_zeta}
\left| \zeta(\xi) - \among{x}{0} \right| 
\leq \left| \xi^{\Lambda} u - \among{x}{0} \right| 
+ \left| \sum_{j \geq 0} \xi^j c_j \left( \tilde{\varphi - \id} \right) (u) \right| 
\end{equation} 
where $\xi^{\Lambda} = (\xi^{\lambda_1}, \ldots, \xi^{\lambda_d})^T$ 
if $\Lambda = \mathrm{Diag}(\lambda_1, \ldots, \lambda_d)$ 
and by convention $\xi^0 = 1$ for all $\xi \in \C$. 
Because $\lambda_{\alpha} = 0$ for all $1 \leq \alpha \leq d_x$,  
$$ 
\left| \xi^{\Lambda} u - \among{x}{0} \right| \leq \left| u - \among{x}{0} \right| \leq \rho_1 + \rho , 
$$ 
and according to the maximum modulus principle 
$$ 
\left| \sum_j \xi^j c_j \left( \tilde{\varphi - \id} \right) (u) \right| 
\leq \sup_{\theta \in \mathbb{T} } | \varphi_{\theta}(u) - u | \leq r - \rho . 
$$ 
The bound 
$ 
\left| \zeta(\xi) - \among{x}{0} \right| \leq \rho_1 + r 
$ 
follows from~\eqref{approx_eq:lemma:bound_zeta}, 
therefore by~\eqref{approx_def:set_K}, 
$\zeta(\xi)$ is in $\K_r$. 

In turn, the function $\xi \mapsto \mathfrak{f}(\xi) := f(\zeta(\xi))$ is well-defined for all $\xi \in \ol{D}$, 
is continuous on this set, and is holomorphic on $D$. 
As such, it can be developed as a power series around $\xi = 0$. 
We write $(\beta_j)$ the coefficients of this power series such that 
for $\xi$ in a neighborhood of $0$, $\beta(\xi) = \sum_{j \in \mathbb Z} \xi^j \beta_j$. 
By Cauchy formula, 
$$ 
\beta_j 
= \frac{1}{2i\pi} \oint_{|\xi| = 1} \xi^{-(j+1)} f \big( \zeta(\xi) \big) \D \xi 
= \frac{1}{2\pi} \int_{-\pi}^{\pi} e^{-i j \theta } f \circ \tilde{\varphi}_{\theta}(u) \D \theta 
= c_j( f \circ \tilde{\varphi}) (u)
, 
$$
therefore 
$ 
f \left( \sum_{j \geq 0} \xi^j c_j ( \tilde{\varphi} ) (u) \right) 
= \sum_{j \in \mathbb Z} \xi^j c_j \left( f \circ \tilde{\varphi} \right) (u) . 
$ 
For $j < 0$, Cauchy's integral theorem ensures that $\beta_j$ vanishes, 
i.e. that $ c_j(f \circ \tilde{\varphi})(u)$ vanishes. 
\end{proof}

Assuming now that $\Phi^{[n]}$ satisfies the assumptions of Lemma~\ref{approx_lemma:commut_cj} 
(this will be proved in Theorem~\ref{approx_thm:diff_cov_bounds}), 
from~\eqref{approx_eq:other_def_eta} we get 
\begin{equation} \label{approx_eq:eta_defect}
\eta^{[n]}_{\tau} 
= \frac{1}{\eps}\left( \dptau + \Lambda \right) \Omega^{[n]}_{\tau} + \dpu \Omega^{[n]}_{\tau} \cdot F^{[n]} - f \circ \Omega^{[n]} . 
\end{equation} 
This means that $\eta^{[n]}$ is indeed a defect, 
and this relation will later serve to prove that $w^{[n]}$ is of size $\bigO(\eps^{n+1})$. 

Before proceeding, given a radius $\rho \in [0,2R]$ 
and a map $(\tau,u)\in \R_+ \times \K_{\rho} \mapsto \psi_{\tau}(u)$, 
let us introduce the norm 
\begin{equation} \label{approx_def:norm_dissip}
\| \psi \|_{\rho} := \sup_{(\tau,u) \in \R_+ \times \K_{\rho}} | \psi_{\tau}(u) | . 
\end{equation}

\begin{lemma} \label{approx-lemma:filt_bnds}
Given a radius $\rho \in [0, 2R]$ and an integer $\nu \in \mathbb N$, let $\varphi$ be a periodic map $(\theta,u) \in \mathbb T \times \K_{\rho} \mapsto \varphi_{\theta}(u)$ 
that is analytic w.r.t. $u$, that is $(\nu + 1)$-times continuously differentiable w.r.t. $\theta$ 
and that has vanishing Fourier coefficients for negative indices, 
i.e. for all $j < 0$, $c_j(\varphi)$ is identically zero. 
Then the associated dissipative map $(\tau,u)\in \R_+ \times \K_{\rho} \mapsto \psi_{\tau}(u)$ defined by 
$$ \psi_{\tau}(u) := \sum_{j\in \mathbb Z} e^{-j\tau} c_j(\varphi)(u) $$
is well defined for $(\tau,u) \in \R_+ \times \K_{\rho}$, analytic w.r.t. $u$ and $\nu$-times continuously differentiable w.r.t. $\tau$. 
Furthermore it respects the following bounds for $0\leq k \leq \nu$, 
$$ 
\big\| \dptau^{\: k} \psi_{\tau} \big\|_{\rho} 
\leq \| \dptheta^{\: k} \varphi\|_{\T, \rho} 
$$
where the norm on $\psi$ and its derivatives is defined by~\eqref{approx_def:norm_dissip}. 
\end{lemma}

\begin{proof} 
It is well-known that the Fourier series of $\varphi$ 
and of its derivatives $\dptheta^{\: k} \varphi $ for $1 \leq k \leq \nu$ 
are absolutely convergent. 
Therefore $\psi_{\tau}(u)$ and $\dptau^{\: k} \psi_{\tau}(u)$ are well-defined for $(\tau,u)$ in $\R_+ \times \K_{\rho}$ by 
$$ \dptau^{\: k} \psi_{\tau}(u) 
= \sum_{j \geq 0} (-j)^{k} e^{-j\tau} c_j(\varphi)(u) 
= i^{\, k} \cdot \sum_{j \geq 0} e^{-j\tau} c_j\left( \dptheta^{\: k} \varphi \right)(u)
.
$$
The absolute convergence also ensures that analyticity w.r.t. $u$ is preserved, 
as an absolutely convergent series of holomorphic functions is holomorphic. 
We define $(\xi,u) \mapsto \zeta_{k}(\xi, u)$ the power series defined for all $\xi \in \C, |\xi| \leq 1$ 
and all~$u\in \K_{\rho}$ by 
$$ \zeta_{k}(\xi, u) = \sum_{j \geq 0} \xi^j c_j \left( \dptheta^{\: k} \varphi \right)(u) $$ 
such that $\dptau^{\: k} \psi_{\tau}(u) = i^{k}\, \zeta_{k}(e^{-\tau}, u)$. 
The maximum modulus principle ensures 
$$ \sup_{\tau \in \R_+} | \dptau^{\: k} \psi_{\tau}(u) | 
\leq \sup_{|\xi| \leq 1} | \zeta_{k}(\xi, u) | 
= \sup_{ |\xi| = 1} | \zeta_{k}(\xi, u) | 
\leq \| \dptheta^{\: k} \varphi \|_{\T, \rho} 
$$
which is the desired result.
\end{proof}

Using the lemma's notations and assumptions, since $\psi$ is $\nu$-times continuously differentiable, 
we may also define the norm 
\begin{equation} \label{approx_def:norm_dissip_p} 
\| \psi \|_{\rho, \nu} := \max_{0 \leq k \leq \nu } \| \dptau^{\: k} \psi\|_{\rho} 
\end{equation}
with $\| \cdot \|_{\rho}$ defined by~\eqref{approx_def:norm_dissip}. 
This result can now be applied to maps $\Omega^{[n]}$ and $\eta^{[n]}$, 
after checking that the Fourier coefficients of the shifted maps 
$\widetilde{\Phi}^{[n]}$ and $\widetilde{\delta}^{[n]}$ vanish for negative indices. 
The shift is given by~\eqref{approx_def:shift_operator}. 

\begin{theorem} \label{approx_thm:diff_cov_bounds}
For $n$ in $\N^*$, let us denote $r_n = R/n$ and $\eps_n := r_n/16 M$ 
with $R$ and $M$ defined in Property~\ref{approx_prop:assump_perio}. 
For all $\eps > 0$ such that $\eps \leq \eps_n$, 
the maps  $(\tau, u) \mapsto \Omega^{[n]}_{\tau}(u)$, 
$u \mapsto F^{[n]}(u)$ and $(\tau, u) \mapsto \eta^{[n]}_{\tau}(u)$ 
given by~\eqref{approx-eq:def_diff_var_chg} and~\eqref{approx-eq:diff_flow} 
are well-defined on $\R_+ \times \K_{R}$ and are analytic w.r.t. $u$. 
The change of variable $\Omega^{[n]}$ and the residue $\eta^{[n]}$ 
are both $(p+1)$-times continuously differentiable w.r.t. $\tau$. 
Moreover, with $\| \cdot \|_{R}$ and $\| \cdot \|_{R, p+1}$ given by~\eqref{approx_def:norm_dissip} and~\eqref{approx_def:norm_dissip_p}, 
the following bounds are satisfied for all $0 \leq \nu \leq p+1$, 
\begin{align*}
(i)\quad{}&{} 
\left\| \Omega^{[n]} - e^{-\tau\Lambda} \right\|_{R} \leq 4\eps M, 
& 
\quad(ii)\quad{}&{} 
\left\| \dptheta^{\:\nu}  \big[ \Omega^{[n]} - e^{-\tau \Lambda} \big] \right\|_{R} \leq 8 \big( 1+\vertiii{\Lambda} \big)^{\nu} \eps M \, \nu! 
\\
(iii)\quad{}&{} 
\| F^{[n]} \|_{R} \leq 2M 
& 
(iv)\quad{}&{} 
\| \eta^{[n]}_{\tau}(u) \|_{R,p+1} \leq 2 M \big( 1 + \vertiii{\Lambda} \big)^{p+1} \left( 2\mathcal Q_p \frac{\eps}{\eps_n} \right)^n 
\end{align*}
where $\vertiii{\cdot}$ is the induced norm from $\R^d$ to $\R^d$, 
and $\mathcal Q_p$ is a $p$-dependent constant. 
\end{theorem} 

\begin{proof}
We show by induction that $\widetilde{\Phi}^{[n]}_{\theta}(u)$ and $\widetilde{\delta}^{[n]}_{\theta}(u)$ only have non-negative Fourier modes. 
To start off the induction, notice $ \Phi^{[0]} = \id$, therefore $\widetilde{\Phi}^{[0]}_{\theta}(u) = e^{i\theta\Lambda}u$. 
Since $\Lambda$ only has coefficients in $\N$, only positive modes are generated. 
Assuming for $0 \leq k < n$ that $\widetilde{\Phi}^{[k]}$ 
has vanishing Fourier coefficients for negative indices, 
let us prove that $\widetilde{\Phi}^{[k+1]}$ does as well.
By definition~\eqref{approx-eq:periodic_iter}, 
$$ \widetilde{\Phi}^{[k+1]}_{\theta}(u) 
= e^{i\theta\Lambda}u 
+ \eps \int_0^{\theta} e^{i\theta\Lambda} T(\Phi^{[k]})_{\sigma}(u) \D\sigma 
- \eps e^{i\theta\Lambda} \left\langle \int_0^{\bullet} T(\Phi^{[k]})_{\sigma}(u) \D\sigma \right\rangle $$
from which we gather that the only problematic term in the definition of $\widetilde{\Phi}^{[k+1]}$ is the integral 
$$ \int_0^{\theta} e^{i\theta\Lambda} T(\Phi^{[k]})_s \D s 
= \int_0^{\theta} e^{i(\theta - s) \Lambda} \tilde{ T(\Phi^{[k]}) }_s \D s
= \int_0^{\theta} e^{i(\theta - s) \Lambda} \left(-i\, f\circ \widetilde{\Phi}_s^{[k]} + \dpu \widetilde{\Phi}_s^{[k]} \cdot G^{[k]} \right) \D s
$$
where we used $ e^{i\theta\Lambda} g_{\theta}\circ \Phi^{[k]}_{\theta} = -i\, f \circ \widetilde{\Phi}^{[k]}_{\theta} $. 
The convolution product of a periodic map $\theta \mapsto \varphi_{\theta}$ with $\theta \mapsto e^{i \theta \Lambda}$ 
generates only one new, nonnegative mode, which is $\theta \mapsto e^{i\theta \Lambda}$. 
By assumption and Theorem~\ref{approx-thm:period_bounds}, 
Lemma~\ref{approx_lemma:commut_cj} is applicable, 
therefore $f \circ \tilde{\Phi}^{[k]}$ only involves nonnegative modes. 
In turn, the same goes for $\tilde{T(\Phi^{[k]})}$ and then for $\tilde{\Phi}^{[k+1]}$. 

This being true, Lemma~\ref{approx-lemma:filt_bnds} is applicable, producing the desired bounds directly. 
The only relationship needed is
$$ \| \dptheta^{\:\nu} \widetilde{\varphi} \|_{\rho} 
= \left\| \sum_{q = 0}^{\nu} \among{\nu}{q} (i\Lambda)^{q} e^{i\theta\Lambda} \ \dptheta^{\:\nu-q} \varphi \right\|_{\rho}
\leq (1+\vertiii{\Lambda} )^\nu \| \varphi \|_{\rho,\nu} . $$
\end{proof}

\section{The micro-macro paradigm} \label{sec:mima}

In this section, we start by denoting $v^{[n]}(t) := \Gamma_t^{[n]} \circ \big( \Omega^{[n]}_0 \big)^{-1}(u_0) $ 
and inject the decomposition 
\begin{equation} \label{mima_eq:diff_decomp}
u\eeps(t) = \Omega_{t/\eps}^{[n]} \big( v^{[n]}(t) \big) + w^{[n]}(t) 
\end{equation}
into the original problem~\eqref{intro-pb:full_pb_on_u} 
in order to find a system on $v^{[n]}$ and $w^{[n]}$. 
The idea of the decomposition is that $v^{[n]}$ and $w^{[n]}$ are not stiff 
and can therefore be computed with \textit{uniform accuracy}, 
i.e. the numerical error is independent of $\eps$. 

With definition~\eqref{mima_eq:diff_decomp}, $w^{[n]}$ is of size $\bigO(\eps^{n+1})$ 
and its derivatives are bounded uniformly up to order~$n+1$. 
This demonstration is the subject of Subsection~\ref{sec:mima:subsec:pb}. 
In Subsection~\ref{sec:mima:subsec:ua}, we prove that using explicit exponential Runge-Kutta of order $n+1$ 
to compute $v^{[n]}$ and $w^{[n]}$ generates an error of uniform order $n+1$ on~$u\eeps$ 
as defined by~\eqref{intro_eq:unif_error}. 

\subsection{Definition and properties of the micro-macro problem}
\label{sec:mima:subsec:pb}

From decomposition~\eqref{mima_eq:diff_decomp} 
we obtain the following system 
$$
\renewcommand{\arraystretch}{1.4}
\left \{
\begin{array}{ll}
\dpt v^{[n]}(t) = F^{[n]}(v^{[n]}) ,
\\ \displaystyle
\dpt w^{[n]}(t) = 
- \inveps \Lambda \left( \Omega^{[n]}_{t/\eps}(v^{[n]}) + w^{[n]} \right) 
+ f\left( \Omega^{[n]}_{t/\eps}(v^{[n]}) + w^{[n]} \right) 
- \frac{\D}{\D t} \Omega_{t/\eps}^{[n]} (v^{[n]} ), 
\end{array}
\right .
$$
with initial conditions $v^{[n]}(0) = \left( \Omega^{[n]}_0 \right)^{-1}(u_0)$ and $w^{[n]}(0) = 0$. 
By definition of $v^{[n]}$ and using identity~\eqref{approx_eq:eta_defect},
\begin{align*} 
\frac{\D}{\D t} \Omega_{t/\eps}^{[n]}(v^{[n]}(t)) {}&{}
= \inveps \dptau \Omega_{t/\eps}^{[n]}(v^{[n]}) + \dpu \Omega_{t/\eps}^{[n]}(v^{[n]}) \cdot F^{[n]}(v^{[n]})
\\ {}&{}
= -\inveps \Lambda \Omega^{[n]}_{t/\eps}(v^{[n]}) + \eta^{[n]}_{t/\eps}(v^{[n]}) + f\big( \Omega^{[n]}_{t/\eps}(v^{[n]}) \big) .
\end{align*}
We finally get the micro-macro problem 
\begin{subequations} \label{mima_pb:mima_f_eta}
\renewcommand{\arraystretch}{1.4}
\begin{empheq}[left=\empheqlbrace]{align}
&\dpt v^{[n]}(t) = F^{[n]}(v^{[n]}) ,
\label{mima_pb:mima_f_eta_v}
\\ \displaystyle
&\dpt w^{[n]}(t) = - \inveps \Lambda w^{[n]} 
+ f\left( \Omega^{[n]}_{t/\eps}(v^{[n]}) + w^{[n]} \right) 
- f\left( \Omega^{[n]}_{t/\eps}(v^{[n]}) \right) 
- \eta^{[n]}_{t/\eps}(v^{[n]}).
\label{mima_pb:mima_f_eta_w}
\end{empheq}
\end{subequations}
with initial conditions $v^{[n]}(0) = \big( \Omega^{[n]}_0 \big) ^{-1}(u_0),$ $w^{[n]}(0) = 0$. 
Assuming that the macro equation~\eqref{mima_pb:mima_f_eta_v} is well-posed, this can be written in a more convenient form, 
\begin{equation} \label{mima_pb:mima_L_S}
\left\{ \renewcommand{\arraystretch}{1.3} \begin{array}{ll}
\dpt v^{[n]} = F^{[n]}(v^{[n]}),
\\ \displaystyle
\dpt w^{[n]} = -\inveps \Lambda w^{[n]} + L^{[n]}(t/\eps, t, w^{[n]})w^{[n]} + S^{[n]}(t/\eps, t),
\end{array} \right. 
\end{equation} 
where $L^{[n]}(\tau,t,w)\: w = f\left(\Omega^{[n]}_{\tau}\circ v^{[n]}(t) + w \right) - f \left( \Omega^{[n]}_{\tau} \circ v^{[n]}(t) \right)$ i.e.
\begin{align*} 
&
L^{[n]}(\tau, t, w) = \int_0^1 \dpu f \left( \Omega^{[n]}_{\tau} \circ v^{[n]}(t) + \mu w \right) \D \mu ,
\\
\text{and}\qquad&
S^{[n]}(\tau,t) = -\eta^{[n]}_{\tau}(v^{[n]}(t)) .
\qquad\qquad
\end{align*}
After showing that problem~\eqref{mima_pb:mima_f_eta_v} is well-posed, 
we shall use both formulations~\eqref{mima_pb:mima_f_eta} and~\eqref{mima_pb:mima_L_S} 
interchangeably depending on the context.

\begin{theorem} \label{mima_thm:well_posedness}
For all $n\in \N^*$, let us define $r_n = R/n$ and $\eps_n := r_n/16 M$, 
with $R$ and $M$ given in~\eqref{approx_def:cst_M}. 
For all $\eps \leq \eps_n$, 
Problem~\eqref{mima_pb:mima_f_eta} is well-posed until some final time $T_n$ independent of $\eps$, 
and the following bounds are satisfied for all $t \in [0,T_n]$ and $0 \leq \nu \leq \min(n,p)$,
\begin{align*}
(i)\quad& 
v^{[n]}(t) \in \K_{R} &
(ii)\quad& 
|w^{[n]}(t)| \leq \frac{R}{4}\left( \frac{\eps}{\eps_n} \right)^{n+1}
\\
(iii)\quad&
|\dpt^{\:\nu} E^{[n]}(t) | = \bigO(\eps^{n-\nu})  &
(iv)\quad&
\| \dpt^{\nu+1} E^{[n]} \|_{L^1} = \bigO( \eps^{n-\nu} )
\end{align*}
where $E^{[n]} = \dpt w^{[n]} + \inveps \Lambda w^{[n]}$.
\end{theorem}

\begin{proof}
This proof is in several parts: first we show that problem~\eqref{mima_pb:mima_f_eta_v} is well-posed, 
and use this result to show that the bound on $w^{[n]}$ is satisfied, 
thereby also proving that~\eqref{mima_pb:mima_f_eta_w} is well-posed. 
Finally we focus on the bounds on $E^{[n]}$. 

Let us set $\varphi(v) = u_0 + v - \Omega^{[n]}_0(u_0 + v)$. 
Using Theorem~\ref{approx_thm:diff_cov_bounds}, if $|v| \leq R/4$ then $|\varphi(v)| \leq R/4$. 
By Brouwer fixed-point theorem, there exists $v^*$ such that $\varphi(v^*) = v^*$, 
i.e. $u^* \in \K_{R/4}$ such that $\Omega^{[n]}_0(u^*) = u_0$. 
Therefore $v^{[n]}(0) := u^* \in \K_{R/4}$. 

Given $t > 0$ and assuming $v^{[n]}(s) \in \K_{R}$ for all $s \in [0,t]$, 
one can bound $v^{[n]}(t)$ using Theorem~\ref{approx_thm:diff_cov_bounds}:  
$$ \left| v^{[n]}(t) - v^{[n]}(0) \right| = \left| \int_0^{t} F^{[n]}\left(v^{[n]}(s) \right) \D s \right| \leq 2Mt . $$
Setting $T_v := \dfrac{3R}{8M}$ ensures $\left| v^{[n]}(t) - v^{[n]}(0) \right| \leq 3R/4$, meaning that for all $t\in [0,T_v]$, $v^{[n]}(t)$ exists and is in $\K_{R}$.
Again from Theorem~\ref{approx_thm:diff_cov_bounds}, we deduce $\Omega_{\tau}^{[n]} \big( v^{[n]}(t) \big) \in \K_{5R/4}$.

Focusing now on $w^{[n]}$ and assuming for all $s\in [0,t],\ |w^{[n]}(s)| \leq R/4$, 
the linear term $L^{[n]}\big( \tau, s, w^{[n]}(s) \big)$ is bounded using a Cauchy estimate: 
$$ 
\left| L^{[n]} \big( \tau, s, w^{[n]}(s) \big) \right| \leq \| \dpu f \|_{3R/2} 
\leq \frac{ \| f \|_{2R} }{2R - \frac32 R} 
\leq \frac{2M}{R} 
$$
using a Cauchy estimate. 
The integral form then gives the bounds 
\begin{align}
\left| w^{[n]}(t) \right| \leq{}&{} 
	\left| \int_0^t e^{\frac{s-t}{\eps}\Lambda} L^{[n]} \big(s/\eps, s, w^{[n]}(s) \big) w^{[n]}(s) \D s 
			+ \int_0^t e^{\frac{s-t}{\eps}\Lambda} S^{[n]}(s/\eps, s) \D s \right| 
\notag
\\
\leq{}&{}
	\int_0^t \frac{2M}{R} \left| w^{[n]}(s) \right| \D s + \left| \int_0^t e^{\frac{s-t}{\eps}\Lambda} S^{[n]}(s/\eps, s) \D s \right| 
\label{mima_eq:bnd_w}
\end{align}
We compute for each component separately, using~\eqref{approx-eq:def_diff_var_chg} the definition of $\eta^{[n]}$ 
and~\eqref{approx_def:shift_operator} the definition of the shift operator, 
writing $\Lambda = \text{Diag}(\lambda_1, \ldots, \lambda_d)$, 
$$ 
\eta^{[n]}_{\tau}(u) 
= \sum_{j \geq 0} e^{-j\tau} c_j \left( \tilde{\delta}^{[n]} \right) (u) 
= \left( \sum_{j \geq - \lambda_k} e^{-(j+\lambda_k) \tau} c_j \left( \delta^{[n]} \right)(u)_k \right) _{1 \leq k \leq d}
$$ 
from which we get, using $c_0(\delta^{[n]}) = \langle \delta^{[n]} \rangle = 0$, 
\begin{align*}
\left| \int_0^t e^{\frac{s-t}{\eps}\Lambda} S^{[n]}(s/\eps, s) \D s \right| 
&= \left| \left( e^{-\lambda_k \frac{t}{\eps}} \sum_{\substack{j + \lambda_k \geq 0 \\ j \neq 0}} \int_0^t e^{-j\frac{s}{\eps}} 
c_j \!\left( \delta^{[n]} \right)\! (v^{[n]}(s))_k \D s
\right)_{1 \leq k \leq d} \right|
\\
&\leq \left| \left( 
\sum_{\substack{ j + \lambda_k \geq 0 \\ j \neq 0}} 
\eps e^{-\lambda_k \frac{t}{\eps}} \left| \frac{1 - e^{-j\frac{t}{\eps}}}{j} \right| \cdot \sup_{u \in \K_R} | c_j(\delta^{[n]})(u)_k |
\right)_{1 \leq k \leq d} \right|
\\
& \leq \eps \cdot \left| \sum_{j \in \mathbb Z^*} \left( \sup_{u \in \K_R} 
\left| \frac{1}{j}\ c_j(\delta^{[n]})(u)_k \right| \right)_{ 1 \leq k \leq d} \right | 
\leq \eps \cdot C \| \delta^{[n]} \|_{\T, R, 1} 
\end{align*}
for some constant $C > 0$ and where~$\| \cdot \|_{\T, R, 1}$ is given by~\eqref{approx_eq:def_norms_perio}. 
We go from the first to the second inequality by bounding the difference of exponentials by $1$. 
Using Theorem~\ref{approx-thm:period_bounds}, 
there exists a constant $M_n > 0$ such that for all $t \in [0, T_v]$, 
\begin{equation}
\left| \int_0^t e^{\frac{s-t}{\eps}\Lambda} S^{[n]}(s/\eps, s) \D s \right| \leq M_n \left( \frac{\eps}{\eps_n} \right)^{n+1} . 
\end{equation}
Using Gronwall's lemma in~\eqref{mima_eq:bnd_w} with this inequality yields 
$$ 
| w^{[n]}(t) | \leq M_n\, e^{\frac{2M}{R} t} \left( \frac{\eps}{\eps_n} \right)^{n+1} \leq M_n\, e^{\frac{2M}{R} t} . 
$$
We now set $T_w > 0$ such that $M_n\, e^{\frac{2M}{R} T_w} \leq R/4$ 
($T_w$ may therefore depend on $n$, but does not depend on $\eps$) and 
$$ T_n = \min ( T_v, T_w ). $$
This ensures the well-posedness of the solution of~\eqref{mima_pb:mima_L_S} on $[0,T_n]$ 
as well as the size of $w^{[n]}$.

Finally, the results on $E^{[n]}$ are a direct consequence of the bounds on the linear term 
$$ 
\sup_{\alpha+\beta+\gamma \leq p+1} \| \dptau^{\alpha} \dpt^{\beta} \dpu^{\gamma} L^{[n]} \| < + \infty 
$$ 
and on the source term 
$$ 
\sup_{0 \leq \alpha + \beta \leq p} \| \dptau^{\alpha} \dpt^{\beta} S^{[n]} \|_{L^{\infty}} = \bigO(\eps^n) , 
\qquad 
\sup_{\substack{\beta \geq 1 \\ 1 \leq \alpha + \beta \leq p+1}} \| \dptau^{\alpha} \dpt^{\beta} S^{[n]} \|_{L^1} = \bigO( \eps^{n+1} ) . 
$$
This stems directly from Cauchy estimates and Theorem~\ref{approx_thm:diff_cov_bounds}. 

\end{proof}

\begin{remark} \label{mima_rq:init_cond} 
So far we have not discussed how to compute the initial condition $v^{[n]}(0)$. 
Setting $\eps \phi^{[n]} = \Phi_0^{[n]} -\id$ and $v_k = v^{[n]}(0)$, 
by definition $\Omega^{[n]}_0(v_n) = \Phi^{[n]}_0(v_n) = u_0$, 
therefore using $v_n = v_{n-1} + \bigO(\eps^n)$ and $\phi^{[n]} = \phi^{[n-1]} + \bigO(\eps^n)$ (see~\cite{cite:mic_mac} for details), 
it is easy to show 
\begin{equation*}
v_n = u_0 - \eps \phi^{[n]}(v_{n-1}) + \bigO(\eps^{n+1}) . 
\end{equation*}
We can now define approached initial conditions for problem~\eqref{mima_pb:mima_L_S} iteratively 
\begin{equation} \label{mima_eq:mima_approx_init}
v_0 = u_0, \qquad v_{n+1} = u_0 - \eps \phi^{[n+1]} \big( v_{k} \big), 
\quad\text{and}\quad
w^{[n]}(0) = u_0 - \Omega^{[n]}_0 \big( v_{n} \big)
\end{equation}
which ensures $w^{[n]}(0) = \bigO(\eps^{n+1})$, meaning our previous results are not jeopardised. 
\end{remark}

\subsection{Uniform accuracy of numerical schemes}
\label{sec:mima:subsec:ua} 
\newcommand{\modnorm}[1]{\left\vert #1 \right\vert _{\eps} }

Using a classic scheme to solve Problem~\eqref{mima_pb:mima_f_eta} 
cannot work due to the term $\frac{1}{\eps}\Lambda w^{[n]}$. 
This is why we focus on exponential schemes, which render this term non-problematic (see~\cite{maset2009unconditional}). 
Furthermore, for these schemes the error bound involves the "modified" norm 
\begin{equation} \label{mima_def:modnorm}
\modnorm{u} = \left | u + \inveps \Lambda u \right | . 
\end{equation}
This norm is interesting because after a short time $t \geq \eps \log(1/\eps)$, 
the $z$-component of the solution $u\eeps$ of~\eqref{intro-pb:full_pb_on_u} is of size~$\eps$, as evidenced by~\eqref{approx_eq:cent_manifold}. 
Using the norm $ \modnorm{\, \cdot \, } $ somewhat rescales $ z\eeps $ (but not $x\eeps$) by $\eps^{-1}$ 
such that studying the error in this norm can be seen as a sort of "relative" error. 
The following theorem uses known results on exponential Runge-Kutta schemes 
which can be found for instance in~\cite{cite:exp_erk_schm, cite:exp_erk}. 

\begin{theorem} \label{mima_thm:uniform}
Under the assumptions of Theorem~\ref{mima_thm:well_posedness} 
and denoting $T_n \leq T$ a final time such that problem~\eqref{mima_pb:mima_f_eta} is well-posed on $[0, T_n]$.  
Given $(t_i)_{i\in [\![0,N ]\!]}$ a discretisation of $[0,T_n]$ of time-step $\Dt := \max_i |\, t_{i+1} - t_i \, |$. 
computing an approximate solution $(v_i, w_i)$ of~\eqref{mima_pb:mima_f_eta} 
using an exponential Runge-Kutta scheme of order $q := \min(n,p)+1$ yields a \textit{uniform} error of order $q$, i.e. 
\begin{equation} \label{mima_eq:unif_abs_err}
\max_{0 \leq i \leq N} \modnorm{ u\eeps(t_i) - \Omega^{[n]}_{t_i/\eps}(v_i) - w_i } \leq C \Dt ^{q} 
\end{equation}
where $C$ is independent of $\eps$. 
\end{theorem}

The left-hand side of this inequality involves $\modnorm{ \, \cdot \, }$ 
and shall be called the modified error. 
It dominates the absolute error which uses $| \cdot |$. 

\begin{proof}
The idea in this proof is to bound the errors on the macro part and micro part separately, using 
$$
\modnorm{ u\eeps(t_i) - \Omega^{[n]}_{t_i/\eps}(v_i) - w_i } 
\leq \modnorm{ \Omega^{[n]}_{t_i/\eps} \left( v^{[n]}(t_i) \right) - \Omega^{[n]}_{t_i/\eps}(v_i) } 
+ \modnorm{ w^{[n]}(t_i) - w_i } . 
$$

As the macro part $v^{[n]}$ involves no linear term, 
the scheme acts like any RK scheme on this part. 
Since~$v^{[n]}$ and~$F^{[n]}$ are non-stiff, the scheme 
is necessarily \textit{uniformly} of order $q$, i.e. 
$$ 
\left| v^{[n]}(t_i) - v_i \right| 
\leq \Dt^{q} \cdot t_i \cdot \| \dpt^{q+1} v^{[n]} \|_{L^{\infty}} 
$$
using usual error bounds on RK schemes. 
The reader may notice that the absolute error involving~$|\cdot|$ was used, 
not the modified error involving~$\modnorm{ \, \cdot \, }$. 
The results in~\cite{cite:exp_erk} 
state that an exponential RK scheme of order $q$ generates an error given by
\begin{equation} \label{mima_eq:err_on_w}
\modnorm{ w^{[n]}(t_i) - w_i } \leq C \Dt^{q} \Big( \| \dpt^{q-1} E^{[n]} \|_{\infty} + \| \dpt^{q} E^{[n]} \|_{L^1} \Big) . 
\end{equation}
The bounds on $E^{[n]} = \dpt w^{[n]} + \frac{1}{\eps} \Lambda w^{[n]}$ and its derivatives 
w.r.t. $\eps$ can be found in Theorem~\ref{mima_thm:well_posedness}, 
rendering the computation of bounds on the error of the micro part straightforward. 
From Theorem~\ref{approx_thm:diff_cov_bounds}.\textit{(i)}, $\Omega^{[n]}_{\tau}(u) = e^{-\tau \Lambda} u + \bigO(\eps)$, 
therefore the error on $\Omega^{[n]}_{t/\eps}(v^{[n]})$ is of the form 
$$ 
\Omega^{[n]}_{t_i/\eps} \big( v^{[n]}(t_i) \big) - \Omega^{[n]}(v_i) 
= e^{-t_i \Lambda/\eps} \big(v^{[n]}(t_i) - v_i \big) + \eps r_i
$$
where $v^{[n]}(t_i) - v_i$ and $r_i$ are of size $t_i \cdot \Dt^q$. 
The error can therefore be bounded, denoting $\vertiii{\cdot}$ the induced norm from $\R^d$ to $\R^d$,  
$$
\modnorm{ \Omega^{[n]}_{t_i/\eps} \big( v^{[n]}(t_i) \big) - \Omega^{[n]}(v_i) }
\leq \left( 1 + \vertiii{ \frac{t_i}{\eps}\Lambda e^{-\frac{t_i}{\eps} \Lambda } } \right) | v^{[n]}(t_i) - v_i | + \left( \eps + \vertiii{\Lambda} \right) | r_i | .
$$
From this we get the desired result on~$u\eeps$.
\end{proof}

\begin{remark} \label{mima_rq:imex}
Only exponential schemes are considered here rather than IMEX-BDF schemes 
which are sometimes preferred (as in~\cite{hu2019uniform}). 
The reason for this is twofold. 

First, the error bounds are generally better for these schemes. 
Indeed, an IMEX-BDF scheme of order $q$ involves the $L^1$ norm of $\dpt^{q+1} w^{[n]}$, 
which is worse than the $L^1$ norm of $\dpt^q E^{[n]}$. 
The former is of size $\bigO(\eps^{n-q})$ while the latter is of size $\bigO(\eps^{n+1-q})$.
This allows the use of schemes of order $n+1$ rather than $n$.

Second, because $\Lambda$ is diagonal, 
the exponentials $e^{-\frac{\Dt}{\eps}\Lambda}$ are easy to compute. 
Therefore there is no computational drawback to exponential schemes. 
\end{remark}

\section{Application to ODEs derived from suitably discretized PDEs}
\label{sec:pde}

In this section, we present some tools to adapt our previously developed method 
to partial differential equations. 
This is done by studying two hyperbolic relaxation systems of the form 
$$ 
\left\{ \begin{array}{l}
\dpt u + \dpx \tilde{u} = 0 \\ \displaystyle
\dpt \tilde{u} + \dpx u = \frac{1}{\eps}(g(u) - \tilde{u})
\end{array} \right .
$$
where $g$ acts either as a differential operator on $u$, or as a scalar value function. 
These two problems may seem similar in theory, 
and the latter actually serves as a stepping stone to treat the former in~\cite{jin1998diffusive, jin2000uniformly}, 
but we will treat them quite differently in practice. 
Because our results may not be valid when using operators, 
we shall only be studying these problems \textit{after} discretizing it them, 
using either Fourier modes or finite volumes. 

Even after discretization, it will be apparent that a direct application of the method is impossible, 
often because of the apparition of a Laplace operator with the "wrong" sign. 
The goal of this section is precisely to present possible workarounds to overcome the problems that appear. 
As such, the computation of maps $\Omega^{[n]},\ F^{[n]}$ and $\eta^{[n]}$ used in~\eqref{mima_pb:mima_f_eta} will not be detailed. 
Should the reader wish to see a more detailed and direct application of our method, 
they can find one in Subsection~\ref{tests_subsec:oscill}.

\subsection{The telegraph equation}
\label{pde_subsec:telegraph}
\newcommand{\xvar}{\rho}
\newcommand{\zvar}{z}

A commonly studied equation in kinetic theory is the one-dimensional Goldstein-Taylor model, 
also known as the telegraph equation (see~\cite{jin1998diffusive, lemou2008asymptotic}, for instance). 
It can be written, for $(t,x) \in [0,T] \times \R/2\pi\mathbb{Z}$
\begin{equation} \label{test_eq:cont_telegraph}
\left\{ \begin{array}{l}
\dpt \xvar + \dpx j = 0 , \\ \displaystyle
\dpt j + \frac{1}{\eps} \dpx \xvar = -\frac{1}{\eps} j ,
\end{array} \right. 
\end{equation}
where $\rho$ and $j$ represent the mass density and the flux respectively. 
Using Fourier transforms in $x$, it is possible to represent a function $v(t,x)$ by
$$ v(t,x) = \sum_{k \in \mathbb Z} v_k(t) e^{ikx} . $$
Considering a given frequency $k \in \mathbb Z$ the problem can be reduced to 
$$ 
\left\{ \begin{array}{l} \displaystyle 
\dpt \xvar_k = - ik j_k ,
\\ \displaystyle
\dpt j_k = -\frac{1}{\eps} \left( j_k + ik \xvar_k \right) .
\end{array} \right .
$$
Treating this problem using our method directly leads to dead-ends, 
therefore we will guide the reader through our reasoning navigating some of these dead-ends. 
This will lead to micro-macro decompositions of orders 0 and 1. 
These struggles can be seen as limitations of our approach, 
however we show that with only slight tweaks, 
it is possible to obtain an error of uniform order $3$ using a classic exponential RK scheme. 
We see this as an encouragement to keep working with this method. 

In order to make a component $-\inveps z$ appear, it would be tempting to set $\zvar_k = j_k + ik \xvar_k$. 
This quantity would verify the following differential equation 
$$ \dpt \zvar_k = - \frac{1}{\eps} \zvar_k + k^2 \zvar_k - ik^3 \xvar_k . $$
Integrating this differential equation gives 
\begin{equation} \label{test_eq:naive_integ_zk}
z_k(t) = \exp \left( -\lambda \frac{t}{\eps} \right) z_k(0) - ik^3 \int_0^t e^{(s-t) \lambda/\eps} \rho_k(s) \D s . 
\end{equation}
where $\lambda = 1 - \eps k^2$. 
Because $\eps \in (0,1]$ and $k \in \mathbb Z$ should not be correlated, 
$\lambda$ can take any value in $(-\infty, 1)$. 
For $\lambda$ negative, this equation is unstable and cannot be solved numerically
using standard tools. 
To overcome this, we consider the stabilized change of variable instead 
\newcommand{\quot}{1 + \alpha \eps k^2}
$$
\zvar_k = j_k + \frac{ik}{\quot}\ \xvar_k
$$
where $\alpha$ is a positive constant which we shall calibrate as the study progresses. 
This is the same change of variable as before up to $\bigO(\eps)$, 
but $ik \xvar_k$ was regularized with an elliptic operator to help with high frequencies. 
The problem to solve becomes 
\begin{equation} \label{pde_pb:telegraph}
\left\{ \begin{array}{l} \displaystyle 
\dpt \xvar_k = -\frac{k^2}{\quot} \xvar_k - ik \zvar_k ,
\\ \displaystyle
\dpt \zvar_k = -\frac{1}{\eps} \zvar_k + \frac{k^2}{\quot} \zvar_k - \frac{ik^3}{\quot} \left( \alpha + \frac{1}{\quot} \right) \xvar_k.
\end{array} \right .
\end{equation} 
As in~\eqref{test_eq:naive_integ_zk}, the growth of $\zvar_k$ is given by 
$e^{-\lambda t/\eps}$ if $\lambda$ is defined by 
$$
\lambda = 1 - \dfrac{\eps k^2}{1 + \alpha \eps k^2} \in \left( 1 - \frac{1}{\alpha}, 1 \right] .
$$
For stability reasons $\lambda$ must be positive, therefore we shall choose $\alpha \geq 1$. 

Let us set $u_k = (\xvar_k, \zvar_k)^T$ and $\Lambda = \text{Diag}(0,1)$ 
such that $\dpt u_k = -\frac{1}{\eps} \Lambda u_k + f(u_k)$ with 
\begin{equation} \label{pde_def:f_telegraph}
\renewcommand{\arraystretch}{2}
f(u) = \among{\displaystyle
-\frac{k^2}{\quot}u_1 - ik u_2
}{\displaystyle
\frac{k^2}{\quot}u_2 - \frac{ik^3}{\quot} \left( \alpha + \frac{1}{\quot} \right) u_1
} .
\end{equation} 
In the upcoming study, we usually prefer the notation $f(\xvar, \zvar)$ rather than $f(u)$ 
so as to keep the distinction between both coordinates clear. 
Assuming $|k| \leq k_{\max}$, it is possible to bound $f(\xvar_k, \zvar_k)$ independently of $k$ and of $\eps$, 
allowing us to apply the method developed in this paper 
in order to approximate every $\xvar_k$ and $j_k$, 
and eventually $\xvar(x,t)$ and $j(x,t)$. 
Note that no rigorous aspects of convergence in functional spaces are considered here 
-- this will be treated in a forthcoming work. 
We omit the index $k$ going forward for the sake of clarity. 

The micro-macro method is initialized by setting the change of variable $\Omega^{[0]}_{\tau}(\xvar, \zvar) = ( \xvar, e^{-\tau} \zvar )^T$. 
The vector field followed by the macro part $v^{[0]}$ is $F^{[0]}$ given by  
\begin{equation} \label{pde_eq:def_kmod}
F^{[0]}(\xvar,\zvar) = \hat{k}^2\among{-\xvar}{\zvar} 
\qquad \text{with} \qquad 
\hat{k} = \frac{k}{\sqrt{\quot}} .
\end{equation} 
This means that the macro variable $v^{[0]}(t)$ is given by 
$$ 
v^{[0]}(t) = \begin{pmatrix} 
e^{ - \hat{k}^2 t } & 0 \\
0 & e^{ \hat k^2 t} \end{pmatrix}
v^{[0]}(0) .
$$
Notice that the growth of $v^{[0]}_2(t)$ is in $e^{\hat{k}^2 t}$, 
which is akin to the heat equation in reverse time.%
\footnote{
This problem does not appear for the oscillatory equivalent~\eqref{approx-pb:periodic}: 
A direct calculation yields $G^{[0]}(y) = i \hat{k}^2 (y_1, -y_2)^T$, 
meaning both components of the macro part in $y\eeps$ oscillate. 
}
This is problematic, as it is possible for $\hat{k}$ to be quite big. 
For example with $k = 10, \alpha = 2$ and $\eps = 10^{-2}$, 
one gets $e^{\hat{k}^2} \approx 3\cdot 10^{14}$. 
In order to obtain the solution of~\eqref{test_eq:cont_telegraph}, 
$u_k(t) = \Omega^{[0]}_{t/\eps} \left( v^{[0]}(t) \right) + w^{[0]}(t)$, 
however, we are only interested in $\Omega^{[0]}_{t/\eps}\left ( v^{[0]}(t) \right)$ 
for the macro part, 
and $\eta^{[0]}_{t/\eps} \left( v^{[0]}(t) \right)$ for the micro part, 
which only depend on $e^{-\frac{t}{\eps}\Lambda} v^{[0]}(t)$ 
as can be seen in the upcoming expression of $\eta^{[0]}$
and using $\Omega^{[0]}_{\tau}(u) = e^{-\tau \Lambda} u$. 
This means that the interesting quantity is 
\begin{equation} \label{test_eq:tel_mac_zero}
e^{-\frac{t}{\eps}\Lambda} v^{[0]}(t) = \begin{pmatrix} 
e^{ - \hat{k}^2 t } & 0 \\
0 & e^{ -(1 - \eps \hat k^2) \frac{t}{\eps}} \end{pmatrix}
v^{[0]}(0) .
\end{equation}
Recognizing $1 - \eps \hat{k}^2 = \lambda > 0$ in this expression, 
it follows that $v^{[0]}_2$ is a decreasing function of time, 
therefore it is bounded uniformly for all~$t$, $k$ and $\eps$. 
Because the exact computation of $e^{-\frac{t}{\eps}\Lambda}v^{[0]}(t)$ 
is readily available, it is used during implementation, 
leaving only $w^{[0]}$ to be computed numerically using ERK schemes. 
Should the reader wish to conduct their own implementation, they should use the defect 
$$ 
\eta^{[0]}_{\tau}(\xvar, \zvar) = \among{
ik e^{-\tau} \zvar
}{ \displaystyle
\hat{k}^2 \left( \alpha + \frac{1}{\quot} \right) ik \xvar
} 
= \eta^{[0]}_0 (\xvar, e^{-\tau} \zvar).
$$
By linearity of $f$, the micro variable $w^{[0]}$ follows the differential equation 
$$
\dpt w^{[0]} = -\frac{1}{\eps}\Lambda w^{[0]} + f(w^{[0]}) - \eta^{[0]}_0\left( e^{-\frac{t}{\eps}\Lambda} v^{[0]}(t) \right), \qquad w^{[0]}(0) = 0. 
$$
The rescaled macro variable $e^{-\frac{t}{\eps}\Lambda} v^{[0]}(t)$ 
is given by relation~\eqref{test_eq:tel_mac_zero} with initial condition $v^{[0]}(0) = u(0) = (\xvar_k(0), \zvar_k(0) )^T$.

Extending our expansion to order 1 is not trivial either. 
Direct application of iterations~\eqref{approx-eq:periodic_iter} yields 
$$
\Omega^{[1]}_{\tau}(\xvar, \zvar) = \among{ \displaystyle
\xvar + \eps ik e^{-\tau} \zvar 
}{ \displaystyle 
\zvar - \eps \hat{k}^2 \left( \alpha + \frac{1}{\quot} \right) ik \xvar 
} 
$$
from which the vector field for the macro part is 
$$
F^{[1]}(\xvar, \zvar) = \hat{k}^2 \left( 1 + \eps k^2 \left( \alpha + \frac{1}{\quot} \right) \right) 
\among{-\xvar}{\zvar} .
$$
Following the same reasoning as before, one should study 
the evolution of the $z$-component of the rescaled macro variable $e^{-\frac{t}{\eps}\Lambda}v^{[1]}(t)$. 
This evolution is in $e^{-\widetilde{\lambda}t/\eps}$ 
where $\widetilde{\lambda} = 1 - \eps \hat{k}^2 \left( 1 + \eps k^2 \left( \alpha + \frac{1}{\quot} \right) \right)$. 
Studying $\widetilde{\lambda}$ as a function of $\eps k^2$ in $\R_+$ shows that 
it is negative for $\eps k^2 > 1$, whatever the value of $\alpha \geq 1$. 

To circumvent this, we replace $\eps$ by $\frac{\eps}{\quot}$ in iterations~\eqref{approx-eq:periodic_iter}. 
This adds terms of order $\eps^2$ in the definition of $\Omega^{[1]}$ 
that do not modify any properties of the micro-macro decomposition 
but it regularises the problem. 
Specifically, we define 
\begin{equation} \label{pde_eq:cov_telegraph}
\renewcommand{\arraystretch}{1.5}
\Omega^{[1]}_0(\xvar, \zvar) = \among{ \displaystyle
\xvar + \frac{\eps}{\quot} ik \zvar
}{ \displaystyle 
\zvar - \frac{\eps}{\quot} \hat{k}^2 \left( \alpha + \frac{1}{\quot} \right) ik \xvar 
} ,
\end{equation}
from which we get the vector field 
$$ 
F^{[1]}(\xvar,\zvar) = \hat{k}^2 \left( 1 + \eps \hat{k}^2 \left( \alpha + \frac{1}{\quot} \right) \right) \among{-\xvar}{\zvar} .
$$
This time also, the identities $\Omega^{[1]}_{\tau} (u) = \Omega^{[1]}_0 ( e^{-\tau \Lambda} u)$ 
and $\eta^{[1]}_{\tau}(u) = \eta^{[1]}_0( e^{-\tau \Lambda} u)$ are satisfied, 
therefore the interesting variable is $e^{-\frac{t}{\eps}\Lambda}v^{[1]}(t)$. 
The quantity dictating its growth is 
$$
\tilde{\lambda} = 1 - \eps \hat{k}^2 \left( 1 + \eps \hat{k}^2 \left( \alpha + \frac{1}{\quot} \right) \right)
$$
which is positive for all $\eps k^2 \in \R_+$ if and only if $\alpha \geq 2$. 
As with the expansion of order 0, the macro variable should be rescaled and computed exactly. 
The micro part $w^{[1]}$ is given by the differential equation 
$$
\dpt w^{[1]} = -\frac{1}{\eps}\Lambda w^{[1]} + f(w^{[1]}) - \eta^{[1]}_0 \left( e^{-\frac{t}{\eps}\Lambda} v^{[1]}(t) \right), 
\qquad w^{[1]}(0) = u_k(0) - \Omega^{[1]}_0 \left( v^{[1]}(0) \right) 
$$
where, writing $\hat{I} = (1 + \alpha \eps k^2)^{-1}$, 
\begin{equation} \label{pde_eq:eta_telegraph}
\renewcommand{\arraystretch}{1.3}
\eta^{[1]}_{\tau}(\xvar, \zvar) = ik \cdot \eps \hat{k}^2 \left( \alpha + \hat{I} \left( 2 + \eps \hat{k}^2 ( \alpha + \hat{I}) \right) \right) 
\among{
e^{-\tau} \zvar
}{
\hat{k}^2 (\alpha + \hat{I}) \xvar
}
\end{equation} 
\begin{equation} \label{pde_def:init_1_telegraph}
\renewcommand{\arraystretch}{1.5}
\quad \text{and} \quad v^{[1]}(0) = \among{
\xvar_k(0) - \eps \hat{I} ik \zvar_k(0)
}{
\zvar_k(0) + \eps \hat{k}^2 (\alpha + \hat{I} ) ik \xvar_k(0)
} .
\end{equation} 
We approached the initial condition using Remark~\ref{mima_rq:init_cond}, 
but an exact computation of the exact initial condition $ \big( \Omega^{[1]}_0 \big)^{-1}(u_0) $ is possible, 
as the map $ u \mapsto \Omega^{[1]}_0(u)$ is linear.

\begin{proposition} \label{pde_prop:telegraph}
Given a maximum frequency $k_{\max} > 0$ and a scalar $\alpha \geq 2$, and assuming $|k| \leq k_{\max}$, 
the solution $u_k$ of problem~\eqref{pde_pb:telegraph} can be decomposed into 
$$
u_k(t) = \Omega^{[1]}_0 \left( e^{-\frac{t}{\eps}\Lambda} v^{[1]}(t) \right) + w^{[1]}(t) 
$$
where $\Omega^{[1]}_0$ is given by~\eqref{pde_eq:cov_telegraph} and $w^{[1]}(t) = \bigO(\eps^2)$. 
The macro component $v^{[1]}$ is given by 
$$ 
e^{-\frac{t}{\eps}\Lambda} v^{[1]}(t) = \begin{pmatrix}
e^{-K^{[1]} t} & 0 \\ 0 & e^{-(1-\eps K^{[1]}) \frac{t}{\eps}} 
\end{pmatrix} 
v^{[1]}(0) 
$$ 
with $K^{[1]} = \hat{k}^2 \left( 1 + \eps \hat{k}^2 \left( \alpha + \frac{1}{\quot} \right) \right) $, 
$\hat{k} = \frac{k}{ \sqrt{1 + \alpha \eps k^2} }$ 
and $v^{[1]}(0)$ is either $\big( \Omega^{[k]}_0 \big) ^{-1} \big( u_k(0) \big)$ 
or its approximation~\eqref{pde_def:init_1_telegraph}. 
The micro component $w^{[1]}$ is the solution to 
$$
\dpt w^{[1]} 
= -\frac{1}{\eps} \Lambda w^{[1]} + f \left( w^{[1]} \right) 
- \eta^{[1]}_0 \left( e^{-\frac{t}{\eps} \Lambda} v^{[1]}(t) \right) ,
\qquad 
w^{[1]}(0) = u_k(0) - \Omega^{[k]}_0 \left( v^{[1]}(0) \right)
$$
with $f$ and $\eta^{[1]}_0$ given respectively by~\eqref{pde_def:f_telegraph} and~\eqref{pde_eq:eta_telegraph}. 
With these definitions, $w^{[1]}$ can be computed with a uniform error of order~$2$ as defined by~\eqref{intro_eq:unif_error}, 
therefore $u_k$ can be computed with a uniform error of order~$2$. 
\end{proposition}

\subsection{Relaxed conservation law}
\label{pde_subsec:conservation}
\newcommand{\Dx}{\Delta x}
\newcommand{\rlxcont}{\tilde{u}}
\newcommand{\rlxdisc}{\tilde{U}}
\renewcommand{\xvar}{U_1}
\renewcommand{\zvar}{U_2}

Our second test case is a hyperbolic problem for $(t,x) \in [0,T] \times \R/2\pi \mathbb Z$, 
\begin{equation} \label{test_eq:cont_hyperb_uv}
\left\{ \begin{array}{l}
\dpt u + \dpx \rlxcont = 0 , \\ \displaystyle
\dpt \rlxcont + \dpx u = \frac{1}{\eps}(g(u) - \rlxcont) ,
\end{array} \right .
\end{equation}
with smooth initial conditions $ u(0,x) $ and $ \rlxcont(0,x) $. 
This is a stiffly relaxed conservation law, as presented in~\cite{jin1995relaxation}. 
In order to proceed, we require the following condition to be met: 
\begin{equation} \label{test_hyp:conserv_stab}
|g'(u)| < 1
\end{equation}
This is a known stability condition when deriving asymptotic expansions for this kind of problem. 

We start by discretising this system in space with $N > 0$ points. 
Going forward, $(x_j)_{j \in \mathbb{Z}/N\mathbb{Z}}$ denotes a fixed uniform discretisation of $\R/2\pi \mathbb Z$, of mesh size $\Dx := 2\pi/N$. 
We define the vectors $U = (u_j)_j, \rlxdisc = (\rlxcont_j)_j$ 
and, given a vector $V = (v_j)_j$ of size $N$, $g(V) = (g(v_j))_j$. 
For simplicity, $u_j(t)$ is the approximation of $u(t,x_j)$, and the same goes for $\rlxcont$. 
We denote $D$ the matrix of centered finite differences and $L$ the classic discrete Laplace operator, which is to say 
$$D V = \left( \frac{1}{2\Dx} (v_{j+1} - v_{j-1}) \right)_j 
\quad\text{and}\quad
L V = \left( \frac{1}{\Dx^2} ( v_{j+1} - 2 v_j + v_{j-1}) \right)_j $$
Using an upwind scheme after diagonalising problem~\eqref{test_eq:cont_hyperb_uv} yields 
\begin{equation} \label{test_eq:disc_hyperb_uv}
\renewcommand{\arraystretch}{2}
\left\{ \begin{array}{l} \displaystyle
\dpt U + D\rlxdisc - \frac{\Dx}{2} LU = 0 , 
\\ \displaystyle
\dpt \rlxdisc + DU - \frac{\Dx}{2} L\rlxdisc  = \frac{1}{\eps}(g(U) - \rlxdisc ) .
\end{array} \right .
\end{equation}
Setting $\xvar = U$ and $\zvar := \rlxdisc - g(\xvar)$, and neglecting the terms involving $L$ for clarity, 
this problem becomes 
\begin{equation} \label{test_pb:conserv} 
\renewcommand{\arraystretch}{2}
\left\{ \begin{array}{l} \displaystyle
\dpt \xvar =
- D\big( \zvar + g(\xvar) \big) , 
\\ \displaystyle
\dpt \zvar = 
-\inveps \zvar + g'(\xvar) D\zvar - T(\xvar) 
\end{array} \right . 
\end{equation} 
where we defined 
$ T(U_1) := DU_1 - g'(U_1) D g(U_1) . 
$
From this, our method can be applied, but precautions must be taken 
in order to avoid having to solve the heat equation in backwards time. 
Therefore we set 
$$ \renewcommand{\arraystretch}{1.5}
\Omega^{[1]}_{\tau}(\xvar,\zvar) = \among{ \displaystyle
\xvar + \eps (1 - 2\eps D^2)^{-1} D\zvar
}{ \displaystyle
e^{-\tau} \zvar - \eps T(\xvar)
} .
$$
Similarly to the manipulations for the telegraph equation, 
we multiplied $\eps$ by $(I_N - 2 \eps D^2)^{-1}$, 
but this time only for the first component. 
Writing $\widetilde{D} = (I_N - 2 \eps D^2)^{-1} D$, 
the associated vector field is 
$$ \renewcommand{\arraystretch}{2}
F^{[1]}(\xvar, \zvar) = \among{ \displaystyle
- Dg(\xvar) + \eps D T(\xvar) 
}{\displaystyle 
g'(\xvar)D\zvar - \eps T'(\xvar)\widetilde{D}\zvar - \eps^2 g''(\xvar) \big( T(\xvar), \widetilde{D}\zvar \big)
} .
$$
As in Subsection~\ref{tests_subsec:oscill}, it is possible to obtain $\Omega^{[0]}$ and $F^{[0]}$ 
by neglecting the terms of order $\eps$ and above in the expressions above. 

\begin{remark}
Remember that for the telegraph equation, the macro variable $v^{[1]}(t)$ 
needed to be rescaled by $e^{-t\Lambda/\eps}$. 
This is not the case here: In the limit $\Dx \rightarrow 0$, 
the macro variable $v^{[1]} = (\ol{u}_1, \ol{u}_2)^T$ is given by 
$$ \newcommand{\macx}{\ol{u}_1}
\newcommand{\macz}{\ol{u}_2}
\renewcommand{\arraystretch}{1.3}
\left\{
\begin{array}{l} \displaystyle
\dpt \macx = - \dpx \left[ g(\macx) - \eps \big( 1 - g'(\macx)^2 \big) \dpx \macx \right] ,
\\ \displaystyle
\dpt \macz = g'(\macx) \dpx \macz - \left( 1 - g'(\macx)^2 \right) \cdot (1 - 2\eps \dpx^2 )^{-1} \eps \dpx^2 \macz + \eps \phi\eeps( \macx, \widetilde{D} \macz ) 
\end{array} \right.
$$
with $\widetilde{D} = (1-2\eps \dpx^2)^{-1}\dpx$ and $\phi\eeps(u_1, u_2) = g''(u_1) \left( 2g'(u_1) - \eps (1-g'(u_1)^2) \dpx u_1 \right) u_2 $. 
The operator $(1 - 2\eps \dpx^2)^{-1} \eps \dpx^2$ is bounded, 
therefore $\ol{u}_2$ is well-defined. 
The equation on $\ol{u}_1$ is a well-known result. 
If $\eps$ was also relaxed in the $\zvar$-component of $\Omega^{[1]}$, 
there might be no need for condition~\eqref{test_hyp:conserv_stab} 
but the result would be different. 
\end{remark}

Obtaining the defects of order 0 and 1 from these expressions presents no difficulty. 
For $\eta^{[1]}$, we separate here the $\xvar$-component and the $\zvar$-component for clarity. 
$$ \eta^{[0]}_{\tau}(\xvar , \zvar ) = \among{
e^{-\tau} D \zvar
}{ \displaystyle
T( \xvar )
} , $$

\begin{subequations}
\newcommand{\DD}{\widetilde{D}}
\begin{equation}
\begin{array}{ll} 
\eta^{[1]}_0(\xvar,\zvar)_{\xvar} = \!\!\! & 
D \big( g(\xvar+\eps \DD W) - g(\xvar) \big) 
+ (D - \DD) \zvar
\\ & \displaystyle 
+\ \eps \DD \Big( g'(\xvar)DW 
- \eps T'(\xvar) \DD W - \eps^2 g''(\xvar) \big( T(\xvar), \DD W \big) \Big) , 
\end{array} 
\end{equation} 
\begin{equation} 
\begin{array}{ll}
\eta^{[1]}_0(\xvar,\zvar)_{\zvar} = \!\!\! &
- \big( g'(\xvar + \eps \DD \zvar) - g'(\xvar) \big) D\zvar 
\\ & \displaystyle 
+\ T(\xvar + \eps \DD \zvar) - T(\xvar) - \eps T'(\xvar) \DD \zvar
\\ & \displaystyle 
+\ \eps g'(\xvar + \eps \DD \zvar) DT(\xvar) - \eps^2 g''(\xvar) \big( \DD \zvar, T(\xvar) \big) 
\\ & \displaystyle 
+\ \eps T'(\xvar) \left( D g(\xvar) - \eps T(\xvar) \right) . 
\end{array} 
\end{equation}
\end{subequations}
The values of $\eta^{[1]}_{\tau}(U_1, U_2)$ can be recovered using the identity 
$$ 
\eta^{[1]}_{\tau}(U_1, U_2) = \eta^{[1]}_0(U_1, e^{-\tau} U_2) .
$$

Note that when using a given scheme, solving a single step is much more costly 
for the micro-macro problem than for the direct problem: 
Not only is the system size doubled, but the functions implicated 
necessitate more computing power to obtain a single value (especially $\eta^{[1]}$, as is apparent here). 
It is therefore plausible to think that our method is best for computing values during the transient phase, 
after which it is possible to solve the original problem with uniform accuracy.

\section{Numerical simulations} \label{sec:tests}

In this section we shall demonstrate our results 
by confirming the theoretical convergence rates 
of exponential Runge-Kutta (ERK) schemes from~\cite{cite:exp_erk_schm}. 
We also use these schemes on the original problem~\eqref{intro-eq:xz_pb}, 
thereby exhibiting the problem of order reduction. 

In Subsection~\ref{tests_subsec:oscill} we study a toy model with some non-linearity that can be found in~\cite{cite:asym_bseries}, 
for which we compute the micro-macro expansion up to order 2. 
In Subsection~\ref{sec:tests:subsec:pde}, we showcase the results of uniform convergence 
for the partial differential equations of Section~\ref{sec:pde}. 
For these, the exact solution shall not take into account the error in space, 
i.e. it will be the solution to the discretized problem. 
Finally in Subsection~\ref{sec:tests:subsec:equilibrium}, we present a surprising numerical result of order gain 
for problems near equilibrium.

\subsection{Oscillating toy problem} \label{tests_subsec:oscill}

We first study an "oscillating" problem presented in~\cite{cite:asym_bseries} 
which demonstrates a possible use of the method when studying non-linear problems:
\begin{equation} \label{test-pb:xz_oscil}
\left\{
\begin{array}{ll} \displaystyle
\dot{x} = (1-z) \begin{pmatrix} 0 & -1 \\ 1 & 0 \end{pmatrix} x \\ \displaystyle
\dot z = -\inveps z + x_1^{\:2}x_2^{\:2} 
\end{array} \right.
\end{equation}
with initial conditions $x_0 = (0.1, 0.7)^T$ and $z_0 = 0.05$, and final time $T = 1$. 
This is of the form $\dpt u = -\inveps \Lambda u + f(u)$ when setting 
$$ 
\Lambda = 
\text{Diag}(0,0,1)
\qquad\text{and}\qquad
f(u) = \begin{pmatrix}
-(1 - u_3) u_2 \\
(1 - u_3) u_1 \\
(u_1 u_2)^2
\end{pmatrix} .
$$
In order to apply averaging techniques described in Subsection~\ref{sec:approx:subsec:oscill}, we set 
$$
g_{\theta}(u) 
= -i e^{-i \theta \Lambda} f \left( e^{i \theta \Lambda} u \right)
= -i \begin{pmatrix}
-u_2 + e^{i\theta} u_2 u_3 \\
 u_1 - e^{i\theta} u_1 u_3 \\
e^{-i\theta} (u_1 u_2)^2
\end{pmatrix} .
$$
By construction, $\Phi^{[0]}_{\theta}(u) = u$ therefore $G^{[0]}(u) = \langle g \rangle(u) = (-u_2, u_1, 0)^T$. 
One then gets the change of variable at order 1 using~\eqref{approx-eq:periodic_iter}, 
$$
\Phi^{[1]}_{\theta}(u) = \begin{pmatrix}
u_1 - \eps e^{i\theta} u_2 u_3 \\
u_2 + \eps e^{i\theta} u_1 u_3 \\
u_3 + \eps e^{-i\theta} (u_1 u_2)^2
\end{pmatrix}, 
\quad \text{yielding} \quad
G^{[1]}(u) = -i \begin{pmatrix}
-\left( 1 - \eps (u_1 u_2)^2 \right) u_2 \\
 \left( 1 - \eps (u_1 u_2)^2 \right) u_1 \\
 2 \eps u_1 u_2 u_3 (u_1^{\: 2} - u_2^{\: 2})
\end{pmatrix} 
$$
with the definition $G^{[1]} = \langle g \circ \Phi^{[1]} \rangle$. 
In order to compute the change of variable of the next order $\Phi^{[2]}$ given by~\eqref{approx-eq:periodic_iter}, 
one must compute the difference $g \circ \Phi^{[1]} - \dpu \Phi^{[1]} \cdot G^{[1]}$, denoted $T(\Phi^{[1]})$. 
This difference also serves to compute the defect~$\delta^{[1]}$ defined by~\eqref{approx_eq:def_G_delta}, 
as this definition can be written $\delta^{[1]} = \inveps \dptheta \Phi^{[1]} - T(\Phi^{[1]})$. 
A direct calculation yields 
$$ T(\Phi^{[1]})_{\theta} (u) = -i \cdot 
\begin{pmatrix}
e^{i\theta} u_3 \left( u_2 + \eps e^{i\theta} u_1 u_3 + 2 \eps^2 u_1 u_2^{\: 2} (u_1^{\: 2} - u_2^{\: 2} ) \right) 
\\
-e^{i\theta} u_3 \left( u_1 - \eps e^{i\theta} u_2 u_3 - 2 \eps^2 u_1^{\: 2} u_2 (u_1^{\: 2} - u_2^{\: 2} ) \right) 
\\ 
\newcommand{\sqr}{^{\: 2}}
e^{-i\theta} \left( U_0
+ \eps U_1 
+ \eps^2 U_2 \right)
\end{pmatrix}
$$
where for clarity we defined 
$$ U_0 = \left( u_1^{\: 2} + \eps^2 e^{2i\theta} (u_2 u_3)^2 \right) 
\left( u_2^{\: 2} + \eps^2 e^{2i\theta} (u_1 u_3)^2 \right) , $$
$$ 
U_1 = - 2 u_1 u_2 (u_1^{\: 2} - u_2^{\: 2} ) \left( 1 - \eps (u_1 u_2)^2 + \eps e^{3 i\theta} u_3^{\: 3} \right) 
\quad \text{and} \quad 
U_2 = - e^{2i\theta} (2 u_1 u_2 u_3)^2 .
$$
Note that $U_0, U_1$ and $U_2$ depend on both~$\eps$ and~$\theta$.

To compute the expansion of order~2 (i.e. the change of variable of $\Phi^{[2]}$ 
and the vector field $G^{[2]}$), 
we truncate terms of order $\eps^2$ and above in $T(\Phi^{[1]})$ 
(which will not impact results of uniform accuracy) 
and integrate it following formula~\eqref{approx-eq:periodic_iter}. 
Identifying the Fourier coefficients of $e^{i\theta\Lambda} \Phi^{[2]}$, 
we obtain $\Omega^{[2]}$. The vector field $F^{[2]}$ is obtained from $G^{[2]}$. 
Distinguishing the $x$- and $z$-components of $u$, this finally yields 
$$ \renewcommand{\arraystretch}{1.5}
\Omega^{[2]}_{\tau}(x,z) = \begin{pmatrix}
x_1 - \eps e^{-\tau} x_2 z - \frac12 \eps^2 e^{-2\tau} z^2 x_1 \\
x_2 + \eps e^{-\tau} x_1 z - \frac12 \eps^2 e^{-2\tau} z^2 x_2 \\
z + \eps (x_1 x_2)^2 - 2 \eps^2 x_1 x_2 (x_1^2 - x_2^2)
\end{pmatrix}, 
$$
$$ \renewcommand{\arraystretch}{1.5}
F^{[2]}(x,z) = \begin{pmatrix}
x_2 (-1 + \eps (x_1 x_2)^2 - 2 \eps^2 x_1 x_2 (x_1^2 - x_2^2) ) \\
x_1 (1 - \eps (x_1 x_2)^2 + 2 \eps^2 x_1 x_2 (x_1^2 - x_2^2) ) 
\\
2 \eps z x_1 x_2 (x_1^2 - x_2^2)
\end{pmatrix} .
$$
The defect $\eta^{[2]}$ is obtained using relation~\eqref{approx_eq:eta_defect} 
or by computing $\delta^{[2]}$ and identifying the Fourier coefficients. 

\begin{remark}
It is possible to find an approximation of the center manifold $x \mapsto \eps h\eeps(x)$ 
by taking the limit $\tau \rightarrow \infty$ of the $z$-component of $\Omega^{[k]}$. 
For example here 
$$ 
\eps h\eeps(x) = \eps (x_1 x_2)^2 - 2 \eps^2 x_1 x_2 (x_1^2 - x_2^2) + \bigO(\eps^3) .
$$
This coincides with the results in~\cite{cite:asym_bseries}. 
\end{remark}

We remind the reader that the problem that is solved at times $(t_i)_{0 \leq i \leq N}$ is 
$$
\renewcommand{\arraystretch}{1.4}
\left \{
\begin{array}{ll}
\dpt v^{[k]}(t) = F^{[k]}(v^{[k]}) ,
\\ \displaystyle
\dpt w^{[k]}(t) = - \inveps \Lambda w^{[k]} + f\left( \Omega^{[k]}_{t/\eps}(v^{[k]}) + w^{[k]} \right) 
- f\left( \Omega^{[k]}_{t/\eps}(v^{[k]}) \right) 
- \eta^{[k]}_{t/\eps}(v^{[k]}) ,
\end{array}
\right .
$$
with $k = 1,2$. 
This yields vectors $(v_i) \approx (v^{[k]}(t_i))$ and $(w_i) \approx (w^{[k]}(t_i))$, 
from which an approximation $u_i \approx u\eeps(t_i)$ is then obtained by setting 
$ u_i = \Omega^{[k]}_{t_i/\eps}(v_i) + w_i $. 
Initial conditions $v^{[k]}(0)$ and $w^{[k]}(0)$ are computed using Remark~\ref{mima_rq:init_cond}. 

The difference $ f\left( \Omega^{[2]}_{t/\eps}(v^{[2]}) + w^{[2]} \right) 
- f\left( \Omega^{[2]}_{t/\eps}(v^{[2]}) \right) $ is computed using 
$$ \newcommand{\dx}{\tilde{x}}
\newcommand{\dz}{\tilde{z}}
f(x+\dx, z+\dz) - f(x,z) = \begin{pmatrix}
-(1 - z) \dx_2 + (x_2 + \dx_2) \dz \\
(1 - z) \dx_1 - (x_1 + \dx_1) \dz \\
\big( x_1 x_2 + (x_1+\dx_1)(x_2 + \dx_2) \big) 
\left( x_1 \dx_2 + \dx_1 x_2 + \dx_1 \dx_2 \right)
\end{pmatrix}
$$
in order to avoid rounding errors due to the size difference between $u$ and $\tilde{u}$. 

\begin{figure}
\vspace*{-12pt}
\includegraphics[width=\textwidth]{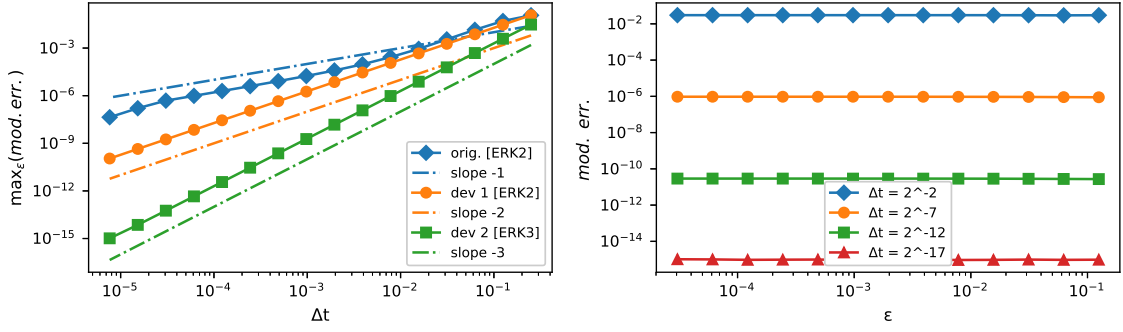}
\vspace*{-24pt}
\caption{Oscillating case: On the left, maximum error on~$\eps$ (for $\eps = 2^{-k}$ with $k$ spanning $\{3, \ldots, 15\}$) as a function of $\Dt$ when using exponential RK schemes (abbr. ERK) of different orders. 
On the right, the error as a function of $\eps$ when solving the micro-macro problem of order~2 using ERK3.}
\label{tests_fig:oscill}
\end{figure}

Figure~\ref{tests_fig:oscill} showcases the phenomenon of order reduction when solving the original problem~\eqref{test-pb:xz_oscil}: 
Despite using a scheme of order 2, the error depends of $\eps$ in such a way that, 
at fixed $\Dt$, there exists no constant $C$ such that the error is bounded by $C\Dt^2$ for all $\eps$. 
However there exists $C$ such that the error is bounded by $C\Dt$. 
This phenomenon of order reduction is discussed in~\cite{cite:exp_erk_schm}. 

In that case, we cannot say that the error is of \textit{uniform} order~1, as this would require the error to be independent of $\eps$. 
However, this is the case when solving the micro-macro problem, 
as can be seen on the right-hand side of Figure~\ref{tests_fig:oscill} for a decomposition of order 2. 
Furthermore, the theoretical orders of convergence from Theorem~\ref{mima_thm:uniform} are confirmed. 
Indeed, using a scheme of order 2 (resp. 3) on the micro-macro problem 
of order 1 (resp. 2) generates a uniform error of the expected order of convergence, 
with no order reduction.

\subsection{Discretized hyperbolic partial differential equations}
\label{sec:tests:subsec:pde}

\paragraph{The telegraph equation \\}
\renewcommand{\xvar}{\rho}
\renewcommand{\zvar}{z}

\begin{figure}
\centering
\includegraphics[width=\textwidth]{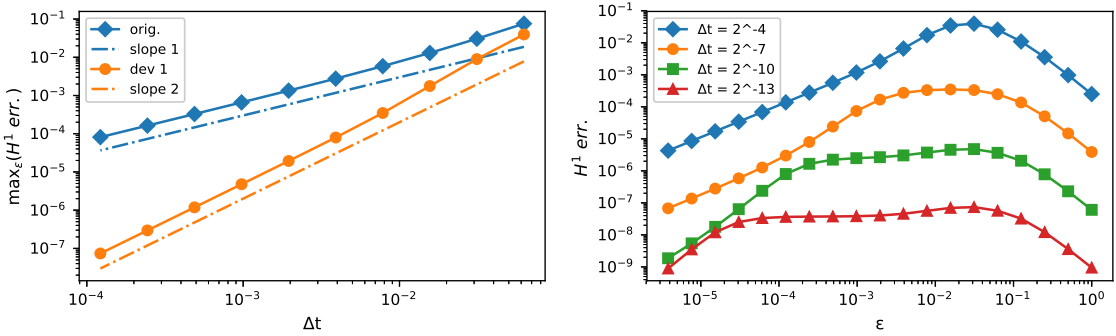}
\vspace*{-24pt}
\caption{Telegraph equation: Absolute $H^1$ error on the solution of~\eqref{test_eq:cont_telegraph} computed by an ERK3 scheme. Supremum on $\eps$ as a function of $\Dt$ (left) and evolution of this error as a function of $\eps$ for the 1st-order decomposition (right). }
\label{test_fig:telegraph}
\end{figure}

Using a spectral decomposition, we solve the problem, for $(t,x) \in [0,T] \times \R/2\pi\mathbb{Z}$, 
\begin{equation*} 
\left\{ \begin{array}{l}
\dpt \xvar + \dpx j = 0 , \\ \displaystyle
\dpt j + \frac{1}{\eps} \dpx \xvar = -\frac{1}{\eps} j , 
\end{array} \right. 
\end{equation*}
by setting $z = j + (1 - \alpha \eps \Delta)^{-1} \dpx z$, yielding problem~\eqref{pde_pb:telegraph}. 
The micro-macro decomposition of order~$1$ is summarized in Property~\ref{pde_prop:telegraph}, 
and its construction is detailed in Subsection~\ref{pde_subsec:telegraph}. 

Implementations are conducted using $\alpha = 2$, 
space frequencies are bounded by $k_{\max} := 12$, 
and initial data is $\xvar(0,x) = e^{\cos(x)}, \ j(0,x) = \frac{1}{2} \cos^3 (x)$. 
Results can be seen in Figure~\ref{test_fig:telegraph} when using a scheme of order 3. 
When solving the original problem, some order reduction is observed, from 3 to 1. 
Here the convergence is not uniform, as it varies with $\eps$ when fixing $\Dt$, 
but this is an artefact due to the exact solving of the macro part: 
The bounds presented in Theorem~\ref{mima_thm:uniform} are \textit{at worst}, 
and the relationship between the error bound and the stiffness of the linear operator 
is rather complex when using exponential RK schemes (again, see~\cite{cite:exp_erk_schm} for details). 
It is known that these schemes have properties of asymptotic preservation (shown in~\cite{dimarco2011exponential} for instance), 
which explains the error variations with $\eps$.

\paragraph{Relaxed conservation law \\} 
\renewcommand{\Dx}{\Delta x}
\renewcommand{\rlxcont}{\tilde{u}}
\renewcommand{\rlxdisc}{\tilde{U}}
\renewcommand{\xvar}{U_1}
\renewcommand{\zvar}{U_2}

\begin{figure}
\centering
\includegraphics[width=\textwidth]{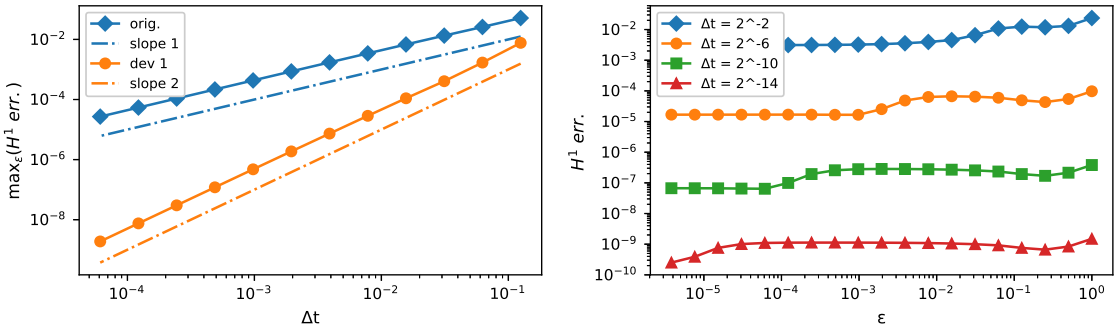}
\vspace*{-24pt}
\caption{Relaxed Burgers-type problem: Maximum modified $H^1$ error (for $\eps$ spanning $1$ to $2^{-18}$ using an ERK3 scheme as a function of $\Dt$ (left), and $H^1$ error as a function of $\eps$ for the micro-macro problem of order~1 (right).}
\label{test_fig:sqr_results}
\end{figure}

Our second test case is a hyperbolic problem for $(t,x) \in [0,T] \times \R/2\pi \mathbb Z$, 
\begin{equation*} 
\left\{ \begin{array}{l}
\dpt u + \dpx \rlxcont = 0 , \\ \displaystyle
\dpt \rlxcont + \dpx u = \frac{1}{\eps}(g(u) - \rlxcont) ,
\end{array} \right .
\end{equation*}
discretized with finite volumes and written in the form of~\eqref{intro-eq:xz_pb} 
by setting $u_1 = u$ and $u_2 = \tilde{u} - g(u)$ 
the $x\eeps$- and the $z\eeps$-component respectively. 
The micro-macro expansion is computed to order~$1$ using the strategy detailed in Subsection~\ref{pde_subsec:conservation}. 

For our tests, following~\cite{hu2019uniform}, we consider 
$ g(u) = bu^2 $ 
with $b = 0.2$. 
Simulations run to a final time $T = 0.25$ and the mesh size is fixed: $N = 16$. 
Initial data is $u(0,x) = \frac{1}{2} e^{\sin(x)}$ and $\rlxcont(0,x) = \cos(x)$. 
The reference solution was computed up to a precision $10^{-12}$ using an ERK2 scheme. 
Convergence results are presented in Figure~\ref{test_fig:sqr_results}, 
confirming theoretical results once more. 

It should be said again that our approach does not study the error in space, only in time. 
For instance, the relationship between the error bound and the grid size is not considered. 
Further studies will be conducted, especially considering CFL conditions, $L^2$ and $H^1$ norms, and computational costs.

\subsection{Near-equilibrium convergence}
\label{sec:tests:subsec:equilibrium}

\begin{figure}
\centering
\includegraphics[width=0.48\textwidth]{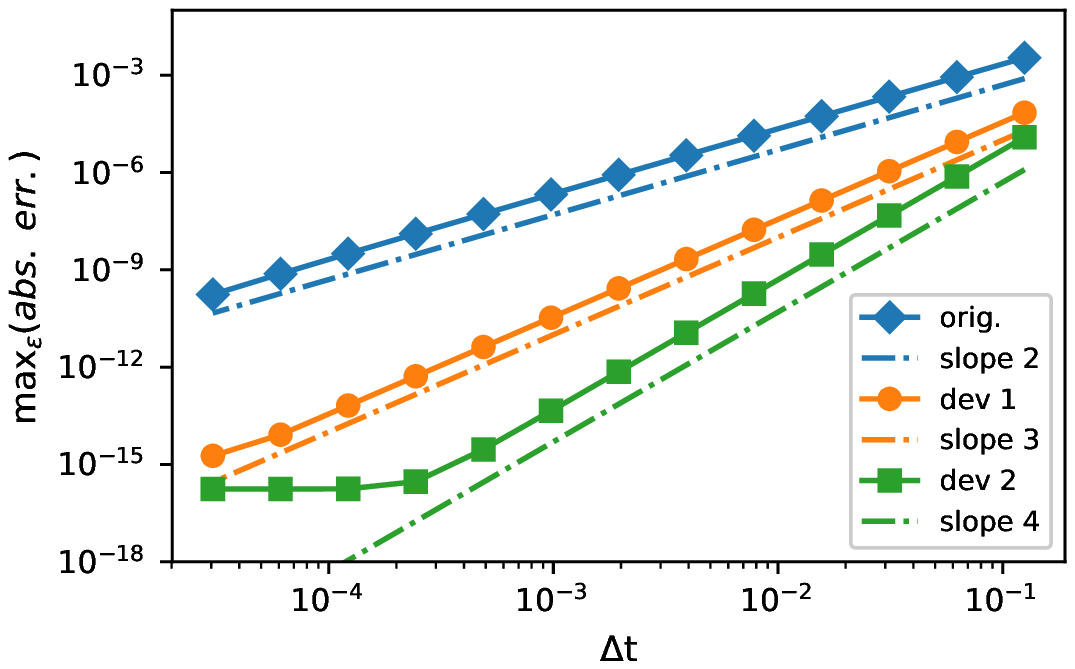}
\includegraphics[width=0.48\textwidth]{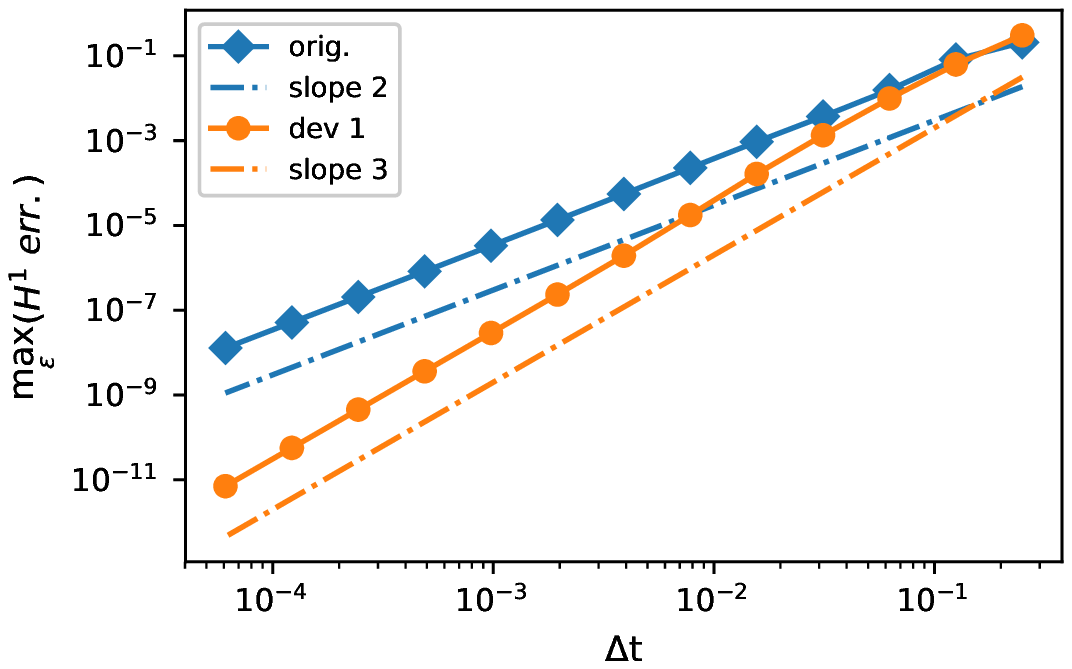}

\vspace*{-10pt}
\includegraphics[width=0.48\textwidth]{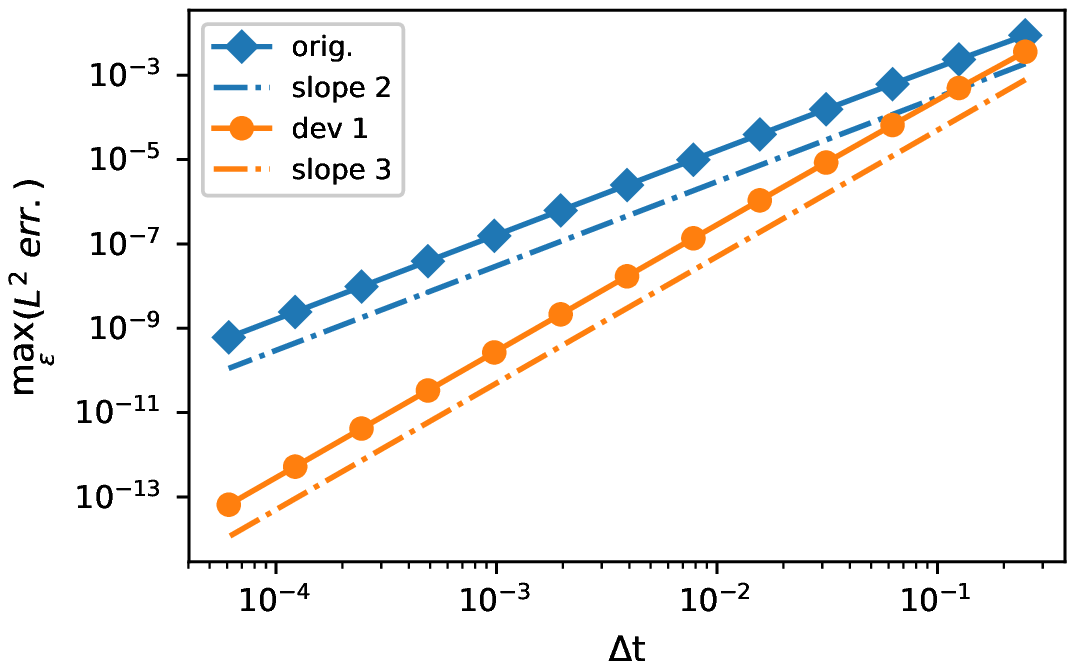}
\vspace*{-10pt}
\caption{In reading order, errors when solving the oscillating toy problem, the telegraph equation and the relaxed conservation law. 
All systems start near equilibrium and are solved with exponential Runge-Kutta schemes of the observed order of convergence. 
}
\label{test_fig:order_gain}
\end{figure}

If one chooses an initial condition $z\eeps(0) = 0$ in~\eqref{intro-eq:xz_pb}, 
then it is close to the center manifold up to $\bigO(\eps)$, 
and Problem~\eqref{intro-pb:full_pb_on_u} can be solved with uniform accuracy of order 2 
but only when considering the absolute error $|\cdot|$, 
not the modified error $\modnorm{\,\cdot\,}$ from~\eqref{mima_def:modnorm}. 
The same behaviour is observed for the telegraph equation when setting $j(0,x) = -\dpx \rho(0,x) $, meaning $z = \bigO(\eps)$. 
This would theoretically mean that we need to push the micro-macro decompositions 
up to order 2 if we want to improve the order of convergence. 
However, this is not the case: 
uniform accuracy of order 3 is obtained from an expansion of order 1 for all test cases. 
This "order gain" also propagates to our micro-macro decomposition of order~2 
for the oscillating toy problem. 
These results can be seen in Figure~\ref{test_fig:order_gain} 
and will be the subject of future works. 


{
\renewcommand{\bibfont}{\small}
\appto{\bibsetup}{\sloppy}
\printbibliography
}

\end{document}